\documentclass[11pt]{article}

\usepackage[utf8]{inputenc}
\usepackage{amsmath, mathtools}
\usepackage{amsfonts, bbm, mathrsfs, yfonts}
\usepackage{amssymb, amsthm, upgreek, bbm}
\usepackage{enumitem, hyperref,xcolor}
\usepackage{subcaption}
\usepackage[sort]{natbib}
\usepackage{geometry}
\usepackage{eufrak}
\usepackage{tikz}

\allowdisplaybreaks

\theoremstyle{plain}
\newtheorem{theorem}{Theorem}
\newtheorem{lemma}[theorem]{Lemma}
\newtheorem{corollary}[theorem]{Corollary}
\newtheorem{proposition}[theorem]{Proposition}

\theoremstyle{definition}
\newtheorem{assumption}{Assumption}
\newtheorem{definition}[theorem]{Definition}

\theoremstyle{remark}
\newtheorem{remark}[theorem]{Remark}
\newtheorem{example}[theorem]{Example}
\newtheorem*{example*}{Rosenbrock example}
\newtheorem{construction}[theorem]{Construction}

\renewcommand{\d}{\mathrm{d}}

\title{Wasserstein de-initialization for Markov chains} 
\author{Mareike Hasenpflug\thanks{University of Passau, Germany, Email: mareike.hasenpflug@uni-passau.de}}

\begin{document}

\maketitle

\begin{abstract}
	This article generalizes the de-initialization framework, proposed by \cite{roberts2001markov} for total variation distance, to Wasserstein distances.
	In essence,  de-initialization captures the phenomenon that
	in the presence of certain structural features, the convergence behaviour of a Markov chain can be reduced to that of a ``simpler'' stochastic process.
	Our results allow to measure this in terms of a general Wasserstein distance.
	We apply them to analyse the Wasserstein convergence behaviour of Gibbs sampling, linchpin variable sampling and slice sampling.
\end{abstract}

\section{Introduction}

Markov chains are ubiquitous in statistics,
not least as Markov chain Monte Carlo (MCMC) methods in the generation of random numbers from a given distribution.
Applications like this make convergence guarantees for Markov chains to their invariant distribution crucial.
The current article is concerned with a phenomenon that is very valuable for understanding the convergence behaviour of a Markov chain if it occurs:
In the presence of certain structural features, the convergence behaviour of a Markov chain can be reduced to that of another stochastic process.
The following toy example may help clarify this phenomenon:

\begin{example*}
	The 2-dimensional Rosenbrock distribution $\pi_{\mathrm{RB}}$ has density 
	\[
		p(x,y) \propto\exp\left (-\frac{1}{20}\big(100(y-x^2)^2 + (1-x)^2\big)\right ), \qquad (x,y) \in \mathbb{R}^2,
	\]
	with respect to the 2-dimensional Lebesgue measure,
	and has a characteristic narrow, bend shape (see Figure \ref{F: Rosenbrock experiment}.a)).
	The marginal distribution of the first coordinate 
	is a 1-dimensional normal distribution $\mathcal{N}(1,10)$ with mean 1 and variance 10,
	and for each $x \in \mathbb{R}$ the conditional distribution of the second coordinate given the first is $\mathcal{N}(x^2,0.1)$.
	This structure makes it easy to generate samples from $\pi_{\mathrm{RB}}$.
	Nonetheless for the sake of demonstration, we consider the following MCMC method targeting the Rosenbrock distribution, which also exploits its structure:
	We run a random walk Metropolis Hastings algorithm with invariant distribution $\mathcal{N}(1, 10)$,
	and then we extend each sample to a full 2-dimensional (approximate) sample from $\pi_{\mathrm{RB}}$ by generating the second coordinate from the conditional distribution, see Figure \ref{F: Linchpin variable sampler for Rosenbrock} for an illustration.
	We denote this MCMC method, which is an instance of a linchpin variable sampler, as $\mathtt{LVS}$.
	More details on linchpin variable samplers can be found in Section \ref{S: Linchpin variable sampler}.
	For now, we compare $\mathtt{LVS}$ to a random walk Metropolis Hastings algorithm $\mathtt{MHA}$ that targets the Rosenbrock distribution directly.
	Examining the trace plots in Figure \ref{F: Rosenbrock experiment}, we observe that $\mathtt{LVS}$ explores the state space much better than $\mathtt{MHA}$.
	On an intuitive level, the reason for this phenomenon may become clear if we reconsider Figure \ref{F: Linchpin variable sampler for Rosenbrock}:
	The ``driving'' Markov chain for $\mathtt{LVS}$ is the 1-dimensional random walk Metropolis Hastings algorithm targeting the isotropic, log-concave distribution $\mathcal{N}(1,10)$,
	whereas $\mathtt{MHA}$ has to operate on the narrow, bend Rosenbrock distribution directly.
	In summary, in contrast to $\mathtt{MHA}$, 
	the convergence behaviour of $\mathtt{LVS}$ seems to be controlled by the simpler 1-dimensional random walk Metropolis Hastings algorithm targeting $\mathcal{N}(1,10)$. 
\end{example*}

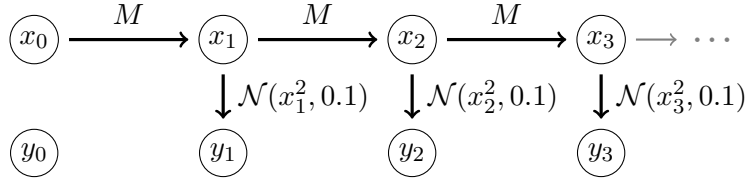
\begin{figure}
	\centering
	\begin{tikzpicture}[inner sep=2pt, outer sep=4pt]
		\coordinate (xnext) at (8.5,0);
		
		\draw (0,0) node[circle,draw] (x0) {$x_0$};
		\draw (0,-1.5) node[circle,draw] (y0) {$y_0$};
		\draw (2.5,0) node[circle,draw] (x1) {$x_1$};
		\draw (2.5,-1.5) node[circle,draw] (y1) {$y_1$};
		\draw (5,0) node[circle,draw] (x2) {$x_2$};
		\draw (5,-1.5) node[circle,draw] (y2) {$y_2$};
		\draw (7.5,0) node[circle,draw] (x3) {$x_3$};
		\draw (7.5,-1.5) node[circle,draw] (y3) {$y_3$};
		
		\path[->, very thick] (x0) edge node[above] {$M$}  (x1) ;
		\path[->, very thick] (x1) edge node[above] {$M$}(x2);
		\path[->, very thick] (x2) edge node[above] {$M$}(x3);
		\path[->,  thick, gray] (x3) edge (xnext);
		\path[->,very thick] (x1) edge node[right] {$\mathcal{N}(x_1^2, 0.1)$}  (y1);
		\path[->, very thick] (x2) edge node[right] {$\mathcal{N}(x_2^2, 0.1)$} (y2);
		\path[->, very thick] (x3) edge node[right] {$\mathcal{N}(x_3^2, 0.1)$} (y3);
		
		\draw (xnext) node[right] {$\color{gray}\boldsymbol{\ldots}$};
	\end{tikzpicture}
	\caption{Graphical illustration of the transition mechanism of $\mathtt{LVS}$ targeting the Rosenbrock distribution $\pi_{\mathrm{RB}}$,
		where $M$ denotes 
		the random walk Metropolis-Hastings algorithm with invariant distribution $\mathcal{N}(1,10)$.}\label{F: Linchpin variable sampler for Rosenbrock}
\end{figure}
\begin{figure}
	\includegraphics[width=\textwidth]{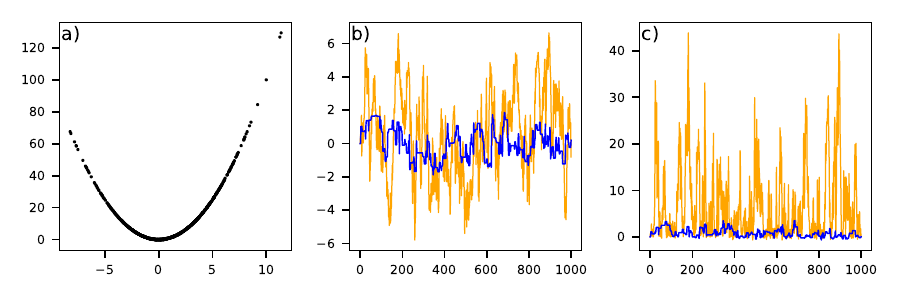}
	\caption{a) Scatter plot of 1.000 samples from the Rosenbrock distribution $\pi_{\mathrm{RB}}$.
		b) \& c) Trace plots of the first and second coordinate of $\mathtt{LVS}$ (orange) and $\mathtt{MHA}$ (blue). Both chains are run for 1000 iterations and initialized at the origin. The proposal distributions of the involved random walk Metropolis Hastings algorithms are Gaussian distributions with variance 1 and the identity as covariance matrix, respectively.}\label{F: Rosenbrock experiment}
\end{figure}

The de-initialization framework proposed by \cite{roberts2001markov} allows to describe this phenomenon in terms of total variation distance.
In our toy example, it yields a total variation distance inequality between $\mathtt{LVS}$ and the 1-dimensional random walk Metropolis Hastings algorithm with invariant distribution $\mathcal{N}(1,10)$.
In other words, de-initialization reveals structural properties of the convergence behaviour.
These can then be exploited to derive upper convergence bounds in total variation distance for the Markov chain of interest,
because one can analyse the (potentially easier to handle) de-initializing stochastic process instead of the original Markov chain.
By reversing the roles it is also possible to obtain lower bounds.

The objective of the current paper is to extend the de-initialization framework to Wasserstein distances,
i.e.\ to formulate conditions when the Wasserstein convergence behaviour of a Markov chain of interest can be reduced to the Wasserstein convergence behaviour of another stochastic process.
We then apply this framework to three different MCMC methods,
linchpin variable sampling, Gibbs sampling and slice sampling.

Wasserstein distances have established as an  important alternative to total variation distance
that is more aware of the topology of the underlying space.
In the  analysis of the convergence behaviour of Markov chains, this is exploited e.g.\  by \cite{mangoubi2018rapid, qu2025computable} and \cite{gibbs2004convergence}, 
which is among the early works on this.
Moreover, in high dimensional settings the use of Wasserstein distances can be beneficial \citep{durmus2015quantitative, qin2022wasserstein, hairer2011asymptotic}.
However, besides extending the applicability, the incorporation of Wasserstein distances into the de-initialization framework also makes it more powerful:
Wasserstein distances are always induced by some cost function, usually a (pseudo) metric, on the underlying state space.
Whereas it is natural to consider the Wasserstein distance for analysing the Markov chain of interest as given,
we are free to choose the Wasserstein distance used to analyse the de-initializing stochastic process.
This gives us a lever to pull to achieve stronger convergence results easier.
We demonstrate this in Section \ref{S: Slice sampling}, where we use this mechanism to derive Wasserstein ergodicity bounds for slice sampling.

As noted by \cite{roberts2001markov}, the concept of de-initialization is related to the notion of sufficient statistics
(see e.g.\ \cite{lauritzen1988extremal, casella1990statistical} and the passage on this topic by \cite{roberts2001markov}),
although they have completely different objectives:
The latter encapsulates the concept that in some respects knowing a sufficient statistics is as informative as knowing all the data,
whereas de-initialization is a tool for understanding the convergence behaviour of Markov chains.
As such, the framework of \cite{roberts2001markov} seems to be  most frequently applied for the total variation distance convergence analysis of Gibbs sampling, where it is used to justify the analysis of related chains in place of the original Gibbs sampler, see e.g.\ \cite{ekvall2021convergence,papaspiliopoulos2020scalable,qin2019convergence,roman2012convergence,roy2007convergence,wang2018geometric,zanella2019scalable}.
Beyond this, total variation distance de-initialization has also been applied to linchpin variable sampling (see \cite{vats2022understanding}),
and its ideas appear in the total variation distance convergence analysis of slice samplers (compare \cite{roberts1999convergence, roberts2002polar}).
Generalizations of the de-initialization framework of \cite{roberts2001markov} beyond total variation distance are not known to the author.

The rest of this paper is organized as follows:
Section \ref{S: Preliminaries} contains some general notation and briefly revises the notions of Wasserstein distance and de-initialization.
Our main results on Wasserstein de-initialization are contained in Section \ref{S: Wasserstein de-initialization}.
In Sections \ref{S: Linchpin variable sampler} and \ref{S: Gibbs sampler}, we show that the Wasserstein convergence behaviour of linchpin variable samplers and Gibbs samplers is equivalent to the behaviour of suitable compound chains,
and in Section \ref{S: Slice sampling}, we use the Wasserstein de-initialization framework to establish Wasserstein ergodicity bounds for slice sampling.
Technical auxiliary results are contained in the appendix.
We use the Rosenbrock example as a running example to illustrate the used concepts on an intuitive level.

\section{Preliminaries}\label{S: Preliminaries}

\subsection{General Notation}

The following section collects some general notation used throughout the paper.		
		
Let $\mathbb{N}$ be the \emph{strictly positive integers} and $\mathbb{N}_0:=\mathbb{N} \cup \{0\}$.

For a Polish space $\mathsf{X}$, we denote by $\mathcal{B}(\mathsf{X})$ its \emph{Borel-$\sigma$-algebra}.
Moreover, we write $\mathcal{M}_1(\mathsf{X}, \mathcal{B}(\mathsf{X}))$ for the \emph{set of all probability measures on $(\mathsf{X}, \mathcal{B}(\mathsf{X}))$}.
Let $(\mathsf{Y}, \mathcal{B}(\mathsf{Y}))$ be a second Polish space.
Unless stated otherwise, the product of two Polish spaces is always assumed to be equipped with the corresponding product topology.
For a probability measure $\mu \in \mathcal{M}_{\mathsf{X} \times \mathsf{Y}}$, we denote by $\mu_{\mathsf{X}} := \mu(\cdot \times \mathsf{Y}) \in \mathcal{M}_{\mathsf{Y}}$ 
and $\mu_{\mathsf{Y}} := \mu(\mathsf{X} \times \cdot ) \in \mathcal{M}_{\mathsf{X}}$ the \emph{marginal distributions} of $\mu$.
The \emph{total variation distance} between two probability measures $\mu, \nu \in \mathcal{M}_{\mathsf{X}}$ is defined as $d_{\mathrm{tv}}(\mu, \nu):= \sup_{\mathsf{A} \in \mathcal{B}(\mathsf{X})} \mu(\mathsf{A}) - \nu(\mathsf{A})$.
Furthermore, we denote the \emph{Dirac measure} for a point $x \in \mathsf{X}$ as $\delta_x$,
and  the $d$-dimensional \emph{Lebesgue measure} on $d$-dimensional Euclidean space $(\mathbb{R}^d, \mathcal{B}(\mathbb{R}^d))$ as $\lambda_d$.

More generally, let $\eta$ be a measure on $(\mathsf{X}, \mathcal{B}(\mathsf{X}))$.
We denote the \emph{restriction} of $\eta$ to a set $\mathsf{A} \in \mathcal{B}(\mathsf{X})$ as $\eta\vert_{\mathsf{A}} := \eta (\cdot \cap \mathsf{A})$.
Moreover, for a measurable function $f: \mathsf{X} \to \mathsf{Y}$, 
the measure $f(\eta)$ is the \emph{image measure} of $\eta \in \mathcal{M}_{\mathsf{X}}$ under $f$ given by $f(\eta)(\mathsf{B}) := \eta(f^{-1}(\mathsf{B}))$, $\mathsf{B} \in \mathcal{B}(\mathsf{Y})$.

If $V: \Omega \to \mathsf{X}$ is a $\mathsf{X}$-valued random variable on some probability space $(\Omega, \mathcal{A}, \mathbb{P})$,
the probability measure $\mathbb{P}^V:= V(\mathbb{P}) \in \mathcal{M}_{\mathsf{X}}$ denotes the \emph{distribution of $V$}.
If $W:\Omega \to \mathsf{Y}$ is a second random variable, we write $\mathbb{P}^{V \mid W}$ for the \emph{conditional distribution of $V$ given $W$},
and $\mathbb{P}^{V \mid W = y}$ for the \emph{conditional distribution of $V$ given $W =y$} for $y \in \mathsf{Y}$. 

A mapping $K: \mathsf{X} \times \mathcal{B}(\mathsf{Y}) \to [0,1]$ is called \emph{Markov} or \emph{transition kernel} if $K(x, \cdot) \in \mathcal{M}_{\mathsf{Y}}$ for all $x \in \mathsf{X}$ and $K(\cdot, \mathsf{B}): \mathsf{X} \to [0,1]$ is measurable for all $\mathsf{B} \in \mathcal{B}( \mathsf{Y})$.
Given $\mu \in \mathcal{M}_{\mathsf{X}}$, we can define a probability measure on $(\mathsf{Y}, \mathcal{B}( \mathsf{Y}))$ by $\mu K := \int_{\mathsf{X}} K(x, \cdot)\ \mu(\d x)$.
Let $H: \mathsf{Y} \times \mathcal{B}( \mathsf{Z}) \to [0,1]$ be a second Markov kernel, where $(\mathsf{Z}, \mathcal{B}(\mathsf{Z}))$ is another Polish space.
Then $KH: \mathsf{X} \times\mathcal{B}(\mathsf{Z}) \to [0,1], (x, \mathsf{C}) \mapsto \int_{\mathsf{Y}} H(y, \mathsf{C})\ K(x, \d y)$ is a Markov kernel, too.
If $\mathsf{Y} = \mathsf{X}$, we also inductively define  $K^1 := K$ and $K^{n+1}:= K^nK$ for $n \in \mathbb{N}$.
Going back to the more general setting of $\mathsf{X}$ not necessarily equal to $\mathsf{Y}$,
let $\mu \otimes K(\cdot) := \int_{\mathsf{X}} \int_{\mathsf{Y}} \mathbbm{1}_{\cdot}(x,y)\ K(x, \d y)\, \mu(\d x) \in \mathcal{M}_{\mathsf{X} \times \mathsf{Y}}$.
Moreover, 
the \emph{tensor product} of $K$ and $H$  is the Markov kernel $K \otimes H: \mathsf{X} \times \mathcal{B}(\mathsf{Y}\times \mathsf{Z}) \to [0,1], (x, \mathsf{D}) \mapsto \int_{\mathsf{Y}} \int_{\mathsf{Z}} \mathbbm{1}_{\mathsf{D}}(y,z)\ H(y, \d z)\, K(x, \d y)$. 
Again, if $\mathsf{X} = \mathsf{Y}$, we can inductively define $K^{\otimes 1} := K$ and $K^{\otimes (n+1)} := K^{\otimes (n+1)} \otimes K$ for $n \in \mathbb{N}$.
If we have a kernel tensor product with exponent smaller equal than zero,
we adopt the convention that the corresponding integral expressions are considered absent.

Furthermore, we denote by $\|\cdot\|_{d}$ the \emph{Euclidean norm} on $(\mathbb{R}^d, \mathcal{B}(\mathbb{R}^d))$.
The index indicating the dimension is sometimes dropped if clear from the context.
Moreover, with an abuse of notation, we also use $\|\cdot\|_{d}$ to denote the induced Euclidean metric.
More generally, for two metrics $\mathrm{d}_{\mathsf{X}}:\mathsf{X}^2 \to [0, \infty)$, $\mathrm{d}_{\mathsf{Y}}: \mathsf{Y}^2 \to [0, \infty)$ their \emph{$\mathfrak{p}$-product metric} is defined as
\[
	\mathrm{d}: (\mathsf{X} \times \mathsf{Y})^2\to [0, \infty), \quad
	\big((x_1,y_1), (x_2,y_2)\big) \mapsto \begin{dcases}
													\big(\mathrm{d}_{\mathsf{X}}(x_1,x_2)^\mathfrak{p} + \mathrm{d}_{\mathsf{Y}}(y_1,y_2)^\mathfrak{p}\big)^{1/\mathfrak{p}}, & \mathfrak{p} \in [1, \infty),\\
													\max\big\{\mathrm{d}_{\mathsf{X}}(x_1,x_2), \mathrm{d}_{\mathsf{Y}}(y_1,y_2)\big\}, & \mathfrak{p} = \infty,
	\end{dcases}
\]
for $\mathfrak{p} \in [1, \infty]$.
Observe that the $\mathfrak{p}$-product metric satisfies for all $x_1,x_2 \in \mathsf{X}, y_1, y_2 \in \mathsf{Y}$
\begin{equation}\label{Eq: product metric ineqaulity}
	\max\big\{\mathrm{d}_{\mathsf{X}}(x_1,x_2), \mathrm{d}_{\mathsf{Y}}(y_1,y_2)\big\}
	\leq \mathrm{d}\big((x_1,y_1), (x_2,y_2)\big)
	\leq \mathrm{d}_{\mathsf{X}}(x_1,x_2) + \mathrm{d}_{\mathsf{Y}}(y_1,y_2),
\end{equation}
where the upper bound follows by Minkowski inequality.

Finally, for two real numbers $r,s \in \mathbb{R}$ we also use $r \wedge s$ to denote their minimum and $r \vee s$ to denote their maximum.
With an abuse of notation, we sometimes write $r \wedge s$ although one of the real numbers is not actually defined, say $s$.
In this case we adopt the convention that $r \wedge s := r$.

\subsection{Wasserstein distances}

Let $(\mathsf{X}, \mathcal{B}(\mathsf{X}))$ be a Polish space equipped with a lower semi-continuous pseudo\footnote{Positive semi-definiteness is allowed to be violated.} metric $c: \mathsf{X} \times \mathsf{X} \to [0, \infty)$,
which we also call \emph{cost function}.
It can be used to induce a notion of distance between probability measures on $(\mathsf{X}, \mathcal{B}(\mathsf{X}))$.
To this end, for $\mu, \nu \in \mathcal{M}_{\mathsf{X}}$ we call
\[
	\mathcal{C}(\mu, \nu) := \Big\{\eta \in \mathcal{M}_{\mathsf{X}^2}\ \big|\ \eta(\mathsf{A} \times \mathsf{X})= \mu(\mathsf{A})\quad \text{and}\quad \eta(\mathsf{X} \times\mathsf{A} )= \nu(\mathsf{A}),\ \forall \mathsf{A} \in \mathcal{B}(\mathsf{X})\Big\}
\]
the \emph{set of couplings} between $\mu$ and $\nu$.
Then for $q \in [1, \infty)$ the \emph{$q$-Wasserstein distance} of $\mu$ and $\nu$ induced by the pseudo metric $c$ is defined as 
\[
	W_{c}^{(q)}(\mu, \nu) := \left(\inf_{\eta \in \mathcal{C}(\mu, \nu)} \int_{\mathsf{X}^2} c^{q}(x,y)\ \eta\big(\d (x,y)\big)\right)^{1/q}.
\]
If $q= 1$, we omit the superscript for the Wasserstein distance.
Interpreting a coupling as a description of how to transform the two marginal distributions into each other,
roughly speaking the $q$-Wasserstein distance records the minimum effort (measured with $c$) to go from $\mu$ to $\nu$.
Observe that there always exists an optimal coupling achieving the infimum in the definition of the Wasserstein distance, see e.g.\ \cite[Theorem 4.1]{villani2009optimal}.

Finally note, the $q$-Wasserstein distance can be even a complete metric on the space of probability measures that satisfy a suitable $q$th moment condition,
if $c$ is a metric that induces the topology on $\mathsf{X}$, see \cite[Section 6]{villani2009optimal}.

\subsection{De-initialization}

In this section we briefly explain the concept of de-initialization and its main implication as proposed by \citep{roberts2001markov},
in the notation and terminology used throughout this paper.

Let $K: \mathsf{X} \times \mathcal{B}(\mathsf{X}) \to [0,1]$ be a Markov kernel on a Polish space $(\mathsf{X}, \mathcal{B}(\mathsf{X}))$. 
This kernel, together with an initial distribution $\mu$ on $(\mathsf{X}, \mathcal{B}(\mathsf{X}))$, specifies the distribution of a  discrete time Markov chain with state space $(\mathsf{X}, \mathcal{B}(\mathsf{X}))$.
By varying the initial distribution we obtain a whole family of Markov chains induced by $K$.
Our aim is to relate the convergence behaviour of these Markov chains to the behaviour of a suitable family of discrete time stochastic processes.
The state space of these stochastic processes does not need to be the same as for our Markov chains of interest.
We assume it is a second Polish space $(\mathsf{Y}, \mathcal{B}(\mathsf{Y}))$.
Like the Markov chains induced by $K$, we index the family of stochastic processes by the initial distributions $\mathcal{M}_{\mathsf{X}}$ on $(\mathsf{X}, \mathcal{B}(\mathsf{X}))$.
More precisely, to describe this family of stochastic processes we use a function 
\begin{equation}\label{Eq: Definition of Pi}
	\Pi: \mathcal{M}_{\mathsf{X}} \times \mathbb{N} \to \mathcal{M}_{\mathsf{Y}}.
\end{equation}
For $\mu \in \mathcal{M}_{\mathsf{X}}$ and $n \in \mathbb{N}$, the probability measure $\Pi(\mu,n)$ on $(\mathsf{Y}, \mathcal{B}( \mathsf{Y}))$ should be understood as the distribution of the $n$th state of the stochastic processes associated with the Markov chain with 
initial distribution $\mu$.

\begin{example*}
	For the Rosenbrock example the setup from the previous paragraph would look as follows:
	The role of the Markov chain of interest with transition kernel $K$ is taken by the $\mathbb{R}^2$-valued linchpin variable sampler $\mathtt{LVS}$ targeting the Rosenbrock distribution.
	As already explained, roughly speaking, instead of the whole chain, it is in fact only the $\mathbb{R}$-valued random walk Metropolis Hastings algorithm  $M$ with invariant distribution $\mathcal{N}(1,10)$ that is relevant for the convergence behaviour of $\mathtt{LVS}$.
	To capture this with the de-initialization framework, $M$ should take the role of the de-initializing stochastic process.
	For the mapping $\Pi$ this means that $(\Pi(\mu,n))_{n \in \mathbb{N}}$ are the distributions of the trajectory of $M$ if $\mathtt{LVS}$ is started with initial distribution $\mu \in \mathcal{M}_{\mathbb{R}^2}$.
\end{example*}

Of course, we cannot expect that any family of stochastic processes governs the convergence behaviour of $K$.
\cite{roberts2001markov} provide a criterion on suitable conditional distributions that ensures this.
To make sense of these conditional distributions, we need random variables that are defined on the same underlying probability space,
such that in the following we will always assume that we can make the following construction.

\begin{construction}\label{A: De-initializing construction}
	Let $(\Omega, \mathcal{A})$ be a measurable space equipped with a sequence of random variables $X_n: \Omega \to \mathsf{X}$, $n \in \mathbb{N}_0$, 
	a sequence of random variables $Y_n: \Omega \to \mathsf{Y}$, $n \in \mathbb{N}$, and a family of probability measures $(\mathbb{P}_\mu)_{\mu \in \mathcal{M}_{\mathsf{X}}}$, such that
	\begin{itemize}
		\item $(X_n)_{n \in \mathbb{N}_0}$ is a Markov chain with transition kernel $K$ and initial distribution $\mu$ under $\mathbb{P}_\mu$, i.e.
		$\mathbb{P}_\mu^{(X_0, \ldots, X_n)} = \mu \otimes K^{\otimes n}$ for all $\mu \in \mathcal{M}_{\mathsf{X}}, n \in \mathbb{N}_0$,
		see e.g.\ \cite[Theorem 1.3.4]{douc2018markov}.
		\item the mapping $\mathsf{X} \times \mathcal{B}(\mathsf{Y}) \to [0,1], (x, \mathsf{B}) \mapsto 		\mathbb{P}^{Y_n}_x(\mathsf{B}):=\mathbb{P}^{Y_n}_{\delta_x}(\mathsf{B})$ is a transition kernel for all $n \in \mathbb{N}$, and
		\begin{equation}\label{Eq: De-initializing process identity}
			\int_{\mathsf{X}} \mathbb{P}_x^{Y_n}(\cdot) \ \mu(\d x) = \Pi(\mu, n), 
			\qquad \forall\ \mu \in \mathcal{M}_{\mathsf{X}}, n \in \mathbb{N}.
		\end{equation}
	\end{itemize}
\end{construction}  

Roughly speaking, this means that we can construct two stochastic processes on the same underlying probability space, 
one of them our Markov chain of interest and the other one described by the mapping $\Pi$, 
and this for all possible initial distribution for our Markov chain of interest.

De-initialization captures now the following concept: 
The random variable $Y_n$ contains enough information about our Markov chain $(X_n)_{n \in \mathbb{N}}$ to specify the distribution of $X_n$ independent of its initial state $X_0$.
The Markov chain $(X_n)_{n \in \mathbb{N}}$ is in this sense ``de-initialized'' by the stochastic process $(Y_n)_{n \in \mathbb{N}}$.
By means of Construction \ref{A: De-initializing construction},
we can express this as a property of the transition kernel $K$ and the mapping $\Pi$ rather than a property of individual Markov chains with fixed initial distribution.

\begin{definition}\label{D: De-initialization}
	Let $(\mathsf{X}, \mathcal{B}(\mathsf{X}))$ and $(\mathsf{Y}, \mathcal{B}(\mathsf{Y}))$ be two Polish spaces.
	We say that a mapping $\Pi: \mathcal{M}_{\mathsf{X}} \times \mathbb{N} \to \mathcal{M}_{\mathsf{Y}}$  is \emph{de-initializing} for a Markov kernel $K: \mathsf{X} \times \mathcal{B}(\mathsf{X})\to [0,1]$ if we can construct objects as in Construction \ref{A: De-initializing construction} such that
	for all $n \in \mathbb{N}$ there exists a Markov kernel $D_n: \mathsf{Y} \times \mathcal{B}(\mathsf{X}) \to [0,1]$ satisfying
	\begin{equation}\label{Eq: De-initialization condition}
		\mathbb{P}_x^{X_n\mid Y_n =\, \cdot } = D_n \quad \mathbb{P}_x^{Y_n}\text{-almost surely}, \qquad \forall\ x \in \mathsf{X}.
	\end{equation}
\end{definition}

\begin{example*}
	In intuitive words, for the explicit setting of the Rosenbrock example the previous definition means:
	Given the current state of the 1-dimensional random walk Metropolis Hastings algorithm $M$ with invariant distribution $\mathcal{N}(1,10)$, 
	the current state of the linchpin variable sampler $\mathtt{LVS}$ targeting $\pi_{\mathrm{RB}}$ must not depend of the initial state of $\mathtt{LVS}$ anymore. 
	The relevant conditional distribution here (taking the role of the kernels $D_n$) is $\mathcal{N}(x^2, 0.1)$, where $x$ is the current state of $M$,
	which indeed does not depend on the initial state of $\mathtt{LVS}$.
	Hence, we can be confident that also the formal definition of de-initialization is satisfied.
\end{example*}

\cite{roberts2001markov} show that if $\Pi$ is de-initializing for $K$, the convergence of $K$ in total variation distance is dominated by $\Pi$.

\begin{theorem}[{\cite[Theorem 1]{roberts2001markov}}]\label{Thm: De-initialization}
	If a Markov kernel $K$ on a Polish space $(\mathsf{X}, \mathcal{B}( \mathsf{X}))$ is de-initialized by a mapping $\Pi: \mathcal{M}_{\mathsf{X}} \times \mathbb{N} \to \mathcal{M}_{\mathsf{Y}}$, where $(\mathsf{Y}, \mathcal{B}(\mathsf{Y}))$ is also a Polish space, then we have
	\[
		d_{\mathrm{tv}}\big(\mu K^n, \nu K^n\big) \leq d_{\mathrm{tv}}\big(\Pi(\mu,n), \Pi(\nu, n)\big),\qquad \forall\ n \in \mathbb{N}, \mu,\nu \in \mathcal{M}_{\mathsf{X}}.
	\]
\end{theorem} 

\begin{example*}
	If we would apply the de-initialization framework to the Rosenbrock example as described above,
	the previous theorem would yield:
	In terms of total variation distance, the convergence of the linchpin variable sampler $\mathtt{LVS}$ targeting $\pi_{\mathrm{RB}}$ is dominated by the convergence of the $\mathbb{R}$-valued random walk Metropolis Hastings algorithm with invariant distribution $\mathcal{N}(0,1)$.
\end{example*}

\section{Wasserstein de-initialization}\label{S: Wasserstein de-initialization}

The main objective of this paper is to generalize the notion of de-initialization suitably to obtain a result like Theorem \ref{Thm: De-initialization} for general Wasserstein distances.
To this end, we equip both state spaces $(\mathsf{X}, \mathcal{B}(\mathsf{X}))$ and $(\mathsf{Y}, \mathcal{B}(\mathsf{Y}))$ with cost functions for the Wasserstein distance.

\begin{assumption}\label{A: Pseudometrics}
	Fix $q \in [1, \infty)$.
	Let $(\mathsf{X}, \mathcal{B}(\mathsf{X}))$ and $(\mathsf{Y}, \mathcal{B}(\mathsf{Y}))$ be two Polish spaces equipped with lower semi-continuous pseudo metrics $\mathrm{d}_{\mathsf{X}}: \mathsf{X} \times \mathsf{X} \to [0, \infty)$ and $\mathrm{d}_{\mathsf{Y}}: \mathsf{Y} \times \mathsf{Y} \to [0, \infty)$.
\end{assumption}

To work for the Wasserstein distances induced by $\mathrm{d}_{\mathsf{X}}$ and $\mathrm{d}_{\mathsf{Y}}$ we need to extend the definition of de-initialization by a suitable Lipschitzness condition on the Markov kernels $D_n$.

\begin{definition}\label{D: Wasserstein de-initialization}
	In the setting of Assumption \ref{A: Pseudometrics},
	we say that a mapping $\Pi: \mathcal{M}_{\mathsf{X}} \times \mathbb{N} \to \mathcal{M}_{\mathsf{Y}}$ is \emph{$q$-Wasserstein de-initializing} (with respect to $\mathrm{d}_{\mathsf{X}}$ and $\mathrm{d}_{\mathsf{Y}}$) for a Markov kernel $K: \mathsf{X} \times \mathcal{B}( \mathsf{X})\to [0,1]$ if $\Pi$ is de-initializing for $K$ and there exists a constant $C \in [0, \infty)$ such that
	for the Markov kernels $D_n$ from Definition \ref{D: De-initialization} holds
	\begin{equation}\label{Eq: Wasserstein de-initializing property}
		W_{\mathrm{d}_{\mathsf{X}}}^{(q)}\big(D_n(y_1, \cdot), D_n(y_2, \cdot)\big)
		\leq C \cdot \mathrm{d}_{\mathsf{Y}}(y_1, y_2), \qquad \forall y_1,y_2 \in \mathsf{Y}, n \in \mathbb{N}.
	\end{equation}
	We say also the $\Pi$ is $q$-Wasserstein de-initializing for $K$ with constant $C$.
\end{definition}

If $\mathrm{d}_{\mathsf{X}}(x_1,x_2) = \mathbbm{1}_{\{x_1 \neq x_2\}}$, $\forall x_1,x_2 \in \mathsf{X}$, and $\mathrm{d}_{\mathsf{Y}}(y_1,y_2) = \mathbbm{1}_{\{y_1 \neq y_2\}}$, $\forall y_1,y_2 \in \mathsf{Y}$,
the Wasserstein distances $W_{\mathrm{d}_{\mathsf{X}}}$ and $W_{\mathrm{d}_{\mathsf{Y}}}$ coincide with the total variation distance.
In this case
\begin{align*}
	d_{\mathrm{tv}}\big(D_n(y_1, \cdot), D_n(y_2, \cdot)\big) \leq
	\begin{rcases}
		\begin{dcases}
			0, & y_1 = y_2\\
			1, & y_1 \neq y_2
		\end{dcases}
	\end{rcases}
	= \mathbbm{1}_{\{y_1 \neq y_2\}}, \qquad \forall y_1,y_2 \in \mathsf{Y}, n \in \mathbb{N},
\end{align*}
trivially holds and we recover the definition of de-initialization.
Therefore, the following theorem covers Theorem \ref{Thm: De-initialization} as a special case.

\begin{theorem}\label{Thm: Wasserstein de-initialization}
	In the setting of Assumption \ref{A: Pseudometrics}, let a Markov kernel $K: \mathsf{X} \times \mathcal{B}( \mathsf{X}) \to [0,1]$ be $q$-Wasserstein de-initialized by a mapping $\Pi: \mathcal{M}_{\mathsf{X}}\times \mathbb{N} \to \mathcal{M}_{\mathsf{Y}}$ with constant $C$.
	Then we have
	\[
		W_{\mathrm{d}_{\mathsf{X}}}^{(q)}\big(\mu K^n, \nu K^n\big)\leq C \cdot W_{\mathrm{d}_{\mathsf{Y}}}^{(q)}\big( \Pi(\mu,n), \Pi(\nu, n)\big), \qquad \forall\ n \in \mathbb{N}, \mu,\nu \in \mathcal{M}_{\mathsf{X}}.
	\]
\end{theorem}

\begin{proof}
	Throughout this proofs we use the objects from Construction \ref{A: De-initializing construction}.
	Fix $n \in \mathbb{N}$.
	By definition of the conditional distribution we have for all $x \in \mathsf{X}$ that
	\begin{equation}\label{Eq: conditional distribution identity}
		K(x, \mathsf{A}) = \mathbb{P}_{x}^{X_n}(\mathsf{A}) = \int_{\mathsf{Y}} \mathbb{P}_{x}^{X_n \mid Y_n = y}(\mathsf{A}) \ \mathbb{P}_x^{Y_n}(\d y)
		\overset{\eqref{Eq: De-initialization condition}}{=} \int_{\mathsf{Y}} D_n(y, \mathsf{A})\ \mathbb{P}_x^{Y_n}(\d y), \qquad \mathsf{A} \in \mathcal{B}(\mathsf{X}).
	\end{equation}
	Fix $\mu, \nu \in \mathcal{M}_{\mathsf{X}}$.
	By \cite[Theorem 1.1]{zhang2000existence}
	there exists a Markov kernel
	$\Xi: \mathsf{Y}^2 \times \mathcal{B}(\mathsf{X}^2) \to [0,1]$
	such that for all $y_1,y_2 \in \mathsf{Y}$ the probability measure $\Xi\big((y_1,y_2), \cdot\big)$ is an optimal coupling of $D_n(y_1, \cdot)$ and $D_n(y_2, \cdot)$ with respect to $W^{(q)}_{\mathrm{d}_{\mathsf{X}}}$.
	Moreover, let $\eta$ be an optimal coupling of $\Pi(\mu,n)$ and $\Pi(\nu,n)$ with respect to $W^{(q)}_{\mathrm{d}_{\mathsf{Y}}}$.
	Then $\eta\, \Xi$ is a coupling of $\mu K^n$ and $\nu K^n$ by virtue of \eqref{Eq: conditional distribution identity}.
	Hence we get from the optimality of $\Xi\big((y_1,y_2), \cdot\big)$ and $\eta$  that
	\begin{align*}
		&W^{(q)}_{\mathrm{d}_{\mathsf{X}}}\big(\mu K^n, \nu K^n\big)
		\leq \left(\int_{\mathsf{X}^2}\mathrm{d}_{\mathsf{X}}(x_1,x_2)^{q}\ \eta\,\Xi\big(\d(x_1, x_2)\big)\right)^{1/q}\\
		&\qquad= \left(\int_{\mathsf{Y}^2} \int_{\mathsf{X}^2} \mathrm{d}_{\mathsf{X}}(x_1, x_2)^{q}\ \Xi\big((y_1,y_2), \d (x_1, x_2)\big)\,\eta\big(\d (y_1, y_2)\big)\right)^{1/q}\\
		&\qquad = \left(\int_{\mathsf{Y}^2} \left(W_{\mathrm{d}_{\mathsf{X}}}^{(q)}\big(D_n(y_1, \cdot), D_n(y_2, \cdot)\big)\right)^{q}\ \eta\big(\d (y_1, y_2)\big)\right)^{1/q}\\
		&\qquad\overset{\eqref{Eq: Wasserstein de-initializing property}}{\leq} C \left(\int_{\mathsf{Y}^2}  \mathrm{d}_{\mathsf{Y}}(y_1,y_2)^{q}\ \eta\big(\d (y_1, y_2)\big)\right)^{1/q}
		= C \cdot W_{\mathrm{d}_{\mathsf{Y}}}^{(q)}\big(\Pi(\mu, n), \Pi(\nu, n)\big).
	\end{align*}
\end{proof}

Because the definitions of de-initialization and Wasserstein de-initialization require Construction \ref{A: De-initializing construction},
the verification of these properties, although conceptually quite intuitive, can still be technically rather tedious.
Therefore, in the following two sections we provide two abstract settings in which we have (Wasserstein) de-initialization and which are easier to verify.
We illustrate their application in Section \ref{S: Applications}.

\subsection{Wasserstein de-initialization through extended lumping}

The first setting we consider is closely related to the concept of lumping.
Especially for finite state spaces and continuous time Markov chains this seems to be a well studied problem, see for example \cite{dynkin1965Markov, kemeny1976finite},
and \cite{rosenblatt1971Markov} for some results for discrete time Markov chains on general state spaces.
The following remark treats the aspects of lumping required in this paper.
\begin{remark}\label{R: Lumping}
	Let $K:\mathsf{X}\times \mathcal{B}(\mathsf{X})\to [0,1]$ be a Markov kernel on a Polish space $(\mathsf{X}, \mathcal{B}(\mathsf{X}))$.
	If there exists a measurable, surjective function $f: \mathsf{X} \to \mathsf{Y}$ onto another Polish space $(\mathsf{Y}, \mathcal{B}( \mathsf{Y}))$ such that
	\[
		K\big(x_1, f^{-1}(\mathsf{A} )\big) = K\big(x_2, f^{-1}(\mathsf{A})\big) \quad \forall\ x_1,x_2 \in \mathsf{X} \text{ with } f(x_1)= f(x_2), \forall \ \mathsf{A} \in \mathcal{B}(\mathsf{Y}),
	\]
	we say that $K$ is \emph{lumping} with respect to $f$.
	This terminology arises from the fact that if $(X_n)_{n\in \mathbb{N}}$ is a Markov chain with kernel $K$ and initial distribution $\mu \in \mathcal{M}_{\mathsf{X}}$,
	then the stochastic process $\big(f(X_n)\big)_{n \in \mathbb{N}}$,
	which ``lumps'' the states in $\mathsf{X}$ that have the same image under $f$,
	is again a Markov chain with kernel
	\[
		Q: \mathsf{Y} \times \mathcal{B}(\mathsf{Y}) \to  [0,1], \qquad (t, \mathsf{A}) \mapsto K\big(x, f^{-1}(\mathsf{A})\big) \text{ for } x \in \mathsf{X} \text{ such that } f(x)= t
	\]
	and initial distribution $f(\mu)$.
	We call $Q$ the \emph{lumped kernel}.
	This follows by
	inductively extending the identity $Q\big(f(x), \mathsf{A}\big) = K\big(x, f^{-1}(\mathsf{A})\big)$, $\ x \in \mathsf{X}, \mathsf{A} \in \mathcal{B}(\mathsf{Y})$ 
	to
	\begin{multline}\label{Eq: Extended relationship lumped and normal kernel}
		\int_{\mathsf{X}^n} h\big(f(x_1), \ldots, f(x_n)\big) \ K^{\otimes n}\big(x, \d (x_1, \ldots, x_n)\big)\\
		= \int_{\mathsf{Y}^n} h(y_1, \ldots, y_n)\ Q^{\otimes n}\big(f(x), \d (y_1, \ldots, y_n)\big), \qquad \forall n \in \mathbb{N}, x \in \mathsf{X}, 
	\end{multline}
	for all measurable functions $h:\mathsf{Y}^n \to [0, \infty)$.
\end{remark}

The following lemma imposes assumptions that imply lumping.
De-initialization in this setting was already considered by \cite[Lemma 1, Proposition 1]{roberts2001markov}.
To make their extension to Wasserstein de-initialization more easily accessible, we repeat these results and a sketch of their proofs in our notation.
\begin{lemma}\label{L: extended lumping}
	Impose Assumption \ref{A: Pseudometrics}, and let $K: \mathsf{X} \times\mathcal{B}(\mathsf{X})\to [0,1]$ be a Markov kernel.
	Assume that there exists a measurable surjective function $f: \mathsf{X} \to \mathsf{Y}$ such that
	\[
		K(x_1, \mathsf{A} ) = K(x_2,\mathsf{A}) \quad \forall\ x_1,x_2 \in \mathsf{X} \text{ with } f(x_1)= f(x_2), \forall \ \mathsf{A} \in \mathcal{B}(\mathsf{X}).
	\]
	Then $K$ is in particular lumping with respect to $f$, see Remark \ref{R: Lumping}. Denote by $Q$ the corresponding lumped kernel,
	and define the Markov kernel
	\[
		R: \mathsf{Y} \times \mathcal{B}(\mathsf{X}) \to [0,1], \qquad (t, \mathsf{B}) 
		\mapsto K\big(x, \mathsf{B}\big) \text{ for } x \in \mathsf{X} \text{ such that } f(x)= t.
	\]
	\begin{enumerate}[label=\alph*)]
		\item The mapping given by $\Pi(\mu, n) := f(\mu)Q^{n-1}$ for $\mu \in \mathcal{M}_{\mathsf{X}}, n \in \mathbb{N}$ is de-initializing for $K$.
			If moreover there exists $C_1 \in [0, \infty)$ such that
			\begin{equation}\label{Eq: condition for upper Wasserstein de-initialization in lumping case}
				W_{\mathrm{d}_{\mathsf{X}}}^{(q)} \big(R(y_1, \cdot), R(y_2, \cdot)\big) \leq C_1 \cdot\mathrm{d}_{\mathsf{Y}}(y_1, y_2)
				\qquad \forall \ y_1,y_2 \in \mathsf{Y},
			\end{equation}
			then $\Pi$ is $q$-Wasserstein de-initializing for $K$ with constant $C_1$.
		\item The Markov kernel $Q$ is de-initialized by the mapping given by $\Pi(\mu, n):= \mu RK^{n-1}$ 
			for $\mu \in \mathcal{M}_{\mathsf{Y}}, n \in \mathbb{N}$. 
			If moreover $f$ is $C_2$-Lipschitz for some $C_2 \in [0, \infty)$, 
			then $\Pi$ is $q$-Wasserstein de-initializing for $Q$ with constant $C_2$.
	\end{enumerate}
\end{lemma}
\begin{proof}
	\textit{Ad a):}
	We describe a suitable Construction \ref{A: De-initializing construction}.
	Let $(X_n)_{n \in \mathbb{N}_0}$ be the canonical Markov chain with transition kernel $K$ on $(\mathsf{X}^{\mathbb{N}_0}, \mathcal{B}(\mathsf{X}^{\mathbb{N}_0}))$ equipped with the corresponding family of probability measures $(\mathbb{P}_\mu)_{\mu \in \mathcal{M}_{\mathsf{X}}}$ (see e.g.\ \cite[Chapter 3.1]{douc2018markov}).
	For all $n \in \mathbb{N}$ set $Y_n = f(X_{n-1})$.
	From Remark \ref{R: Lumping} we know that under $\mathbb{P}_x$ the stochastic process $(Y_n)_{n \in \mathbb{N}}$ is a Markov chain with kernel $Q$ and initial distribution $\delta_{f(x)}$ for all $x \in \mathsf{X}$,
	and we are in the setting of Construction \ref{A: De-initializing construction} with $\Pi(\mu, n) = f(\mu)Q^{n-1}$ for $\mu \in \mathcal{M}_{\mathsf{X}}, n \in \mathbb{N}$.
	
	For all $n \in \mathbb{N}$ and all $x \in \mathsf{X}$, 
	we have $\mathbb{P}_x^{X_n \mid Y_n = y} = R(y, \cdot)$ for $\mathbb{P}_x^{Y_n}$-almost all $y \in \mathsf{Y}$ by virtue of
	\begin{align*}
		&\mathbb{E}_x\big(\mathbbm{1}_{\mathsf{A}}(Y_n) \mathbbm{1}_{\mathsf{B}}(X_n)\big)
		= \mathbb{E}_x\Big[ \mathbb{E}_x \Big(\mathbbm{1}_{\mathsf{A}}\big(f(X_{n-1})\big) \mathbbm{1}_{\mathsf{B}}(X_n) \mid X_{n-1}\Big)\Big]\\
		&\qquad = \mathbb{E}_x\Big[ \mathbbm{1}_{\mathsf{A}}\big(f(X_{n-1})\big) K(X_{n-1}, \mathsf{B})\Big]
		= \mathbb{E}_x\big[ \mathbbm{1}_{\mathsf{A}}(Y_n) R(Y_n, \mathsf{B})\big], \qquad \forall \mathsf{A} \in \mathcal{B}(\mathsf{Y}), \mathsf{B} \in \mathcal{B}(\mathsf{X}),
	\end{align*}
	where we also used the assumption $R\big(f(z), \cdot\big) = K(z, \cdot)$ for $z \in \mathsf{X}$.
	In other words, Definition \ref{D: De-initialization} is satisfied with $D_n = R$ for all $n \in \mathbb{N}$.
	Consequently, if \eqref{Eq: condition for upper Wasserstein de-initialization in lumping case} holds, 
	then also Definition \ref{D: Wasserstein de-initialization} is satisfied.

	\textit{Ad b):}
	We describe a suitable Construction \ref{A: De-initializing construction} with interchanged roles of $\mathsf{X}$ and $\mathsf{Y}$.
	By Kolmogorov's extension theorem for all $\mu \in \mathcal{M}_{\mathsf{Y}}$ there exists a probability measure $\mathbb{P}_\mu$ on $(\mathsf{Y}^{\mathbb{N}_0}\times \mathsf{X}^{\mathbb{N}}, \mathcal{B}(\mathsf{Y}^{\mathbb{N}_0}\times\mathsf{X}^{ \mathbb{N}}))$
	with finite dimensional distributions 
	\begin{align*}
		\eta_{\mathsf{J}_1, \mathsf{J}_2}(\mathsf{A}):= \begin{dcases}
			\begin{multlined}[b]
				\int_{\mathsf{Y}}\int_{\mathsf{X}}\int_{\mathsf{X}^{(j_m \vee k_\ell)-1}}
				\mathbbm{1}_{\mathsf{A}}\big(y_0,f(x_{j_2}), \ldots, f(x_{j_m}), x_{k_1}, \ldots, x_{k_\ell}\big)\\
				\times K^{\otimes (j_m \vee k_\ell)-1}\big(x_1, \d(x_2, \ldots, x_{j_m \vee k_\ell})\big)\, R(y_0, \d x_1)\, \mu(\d y_0),
			\end{multlined} & j_1 = 0,\\
			\begin{multlined}[b]
				\int_{\mathsf{Y}}\int_{\mathsf{X}}\int_{\mathsf{X}^{(j_m \vee k_\ell)-1}}
				\mathbbm{1}_{\mathsf{A}}\big(f(x_{j_1}), \ldots, f(x_{j_m}), x_{k_1}, \ldots, x_{k_\ell}\big)\\
				\times K^{\otimes (j_m \vee k_\ell)-1}\big(x_1, \d(x_2, \ldots, x_{j_m \vee k_\ell})\big)\, R(y_0, \d x_1)\, \mu(\d y_0),
			\end{multlined} & j_1 \neq 0,
		\end{dcases}
	\end{align*}
	for $\mathsf{A} \in \mathcal{B}(\mathsf{Y}^{m} \times\mathsf{X}^{\ell})$ and  $\mathsf{J}_1:=\{j_1, \ldots, j_m\} \subseteq \mathbb{N}_0, \mathsf{J}_2 := \{k_1, \ldots, k_\ell\}\subseteq \mathbb{N}$ finite and assumed to be denoted in increasing order.
	For $n \in \mathbb{N}_0$ let $Y_n : \mathsf{Y}^{\mathbb{N}_0}\times \mathsf{X}^{\mathbb{N}} \to \mathsf{Y}$ be the projection onto the $n$th coordinate of $\mathsf{Y}^{\mathbb{N}_0}$,
	and for $n \in \mathbb{N}$ let  $X_n : \mathsf{Y}^{\mathbb{N}_0}\times \mathsf{X}^{\mathbb{N}} \to \mathsf{X}$ be the projection onto the $n$th coordinate of $\mathsf{X}^{\mathbb{N}}$.
	Using \eqref{Eq: Extended relationship lumped and normal kernel} and $R\big(y, f^{-1}(\cdot)\big) = Q(y, \cdot)$ for $y \in \mathsf{Y}$,
	we get $\mathbb{P}_\mu^{(Y_0, \ldots Y_n)} = \mu \otimes Q^{\otimes n}$ for all $\mu \in \mathcal{M}_{\mathsf{Y}}$, $n \in \mathbb{N}_0$.
	Moreover, $\mathbb{P}_x^{X_n} = RK^{n-1}(x, \cdot)$ for all $x \in \mathsf{X}$, $n \in \mathbb{N}$,
	and we are in the setting of Construction \ref{A: De-initializing construction} with reversed roles of $\mathsf{X}$ and $\mathsf{Y}$.
	
	For $n \in \mathbb{N}$ and $y \in \mathsf{Y}$, we have $\mathbb{P}_y^{Y_n \mid X_n = x} = \delta_{f(x)}$ for $\mathbb{P}_y^{X_n}$-almost all $x \in \mathsf{X}$, by virtue of
	\begin{align*}
		\mathbb{E}_y\big(\mathbbm{1}_{\mathsf{B}}(X_n)\mathbbm{1}_{\mathsf{A}}(Y_n)\big)
		&= \int_{\mathsf{X}}\int_{\mathsf{X}^{n-1}}
			\mathbbm{1}_{\mathsf{A}\times \mathsf{B}}\big(f(x_n), x_n\big)
			K^{n-1}\big(x_1, \d(x_2, \ldots, x_n)\big)\, R(y, \d x_1)\\
		&= \mathbb{E}_y\Big(\mathbbm{1}_{\mathsf{A} }\big(f(X_n)\big)\mathbbm{1}_{\mathsf{B}}(X_n)\Big),
		\qquad \mathsf{A} \in \mathcal{B}(\mathsf{Y}), \mathsf{B} \in \mathcal{B}( \mathsf{X}).
	\end{align*}
	Hence, Definition \ref{D: De-initialization} is satisfied with $D_n = \delta_{f(\cdot)}$.
	
	Since the only coupling between two Dirac measures is their product measure,
	we have
	\begin{equation}\label{Eq: wasserstein distance of Dirac measures}
		W_{\mathrm{d}_{\mathsf{Y}}}^{(q)}\big(\delta_{f(x_1)}, \delta_{f(x_2)}\big) = \mathrm{d}_{\mathsf{Y}}\big(f(x_1), f(x_2)\big),
		\qquad x_1,x_2 \in \mathsf{X}.
	\end{equation}
	Consequently, if $f$ is Lipschitz, then also Definition \ref{D: Wasserstein de-initialization} is satisfied.
\end{proof}

\subsection{Wasserstein de-initialization through splitting}

We turn to the second general setting in which we consider de-initialization.

\begin{lemma}\label{L: De-initialization via splitting}
	In the setting of Assumption \ref{A: Pseudometrics}, let $K: \mathsf{X} \times \mathcal{B}(\mathsf{X}) \to [0,1]$ be a Markov kernel.
	Assume that there exists two Markov kernels $H_1: \mathsf{X} \times \mathcal{B}(\mathsf{Y}) \to [0,1]$ and $H_2: \mathsf{Y} \times \mathcal{B}(\mathsf{X}) \to [0,1]$ such that $K=H_1H_2$. Set $\tilde{K}:= H_2H_1$.
	\begin{enumerate}[label=\alph*)]
		\item The Markov kernel $K$ is de-initialized by the mapping given by $\Pi(\mu,n) = \mu H_1\tilde{K}^{n-1}$ 
			for $\mu \in \mathcal{M}_{\mathsf{X}}$, $n \in \mathbb{N}$.
			If there exists $C_1 \in[0, \infty)$ such that
			\begin{equation}\label{Eq: upper Wasserstein-deinitialization condition splitting case}
				W_{\mathrm{d}_{\mathsf{X}}}^{(q)}\big(H_2(y_1,\cdot), H_2(y_2, \cdot)\big) \leq C_1 \cdot \mathrm{d}_{\mathsf{Y}}(y_1,y_2), \qquad \forall y_1,y_1 \in \mathsf{Y},
			\end{equation}
			then $\Pi$ is $q$-Wasserstein de-initializing for $K$ with constant $C_1$.
		\item If there exists a measurable function $f: \mathsf{X} \to \mathsf{Y}$ such that for all $y \in \mathsf{Y}$ holds
			\begin{equation}\label{Eq: Fiber property}
				f(x) = y \qquad \text{for } H_2(y, \cdot)\text{-almost all } x \in \mathsf{X},
			\end{equation}
			then the mapping given by $\Pi(\mu, n) = \mu H_2 K^n$ for $\mu \in \mathcal{M}_{\mathsf{Y}}$, $n \in \mathbb{N}$ is de-initializing for $\tilde{K}$.
			If in addition $f$ is $C_2$-Lipschitz for some $C_2 \in [0, \infty)$,
			then $\Pi$ is $q$-Wasserstein de-initializing for $\tilde{K}$ with constant $C_2$.
	\end{enumerate} 
\end{lemma}

\begin{proof}
	\textit{Ad a):}
	First, we provide a suitable Construction \ref{A: De-initializing construction}.
	For $\mu \in \mathcal{M}_{\mathsf{X}}$ let $\mathbb{P}_\mu$ be the probability measure on 
	$\big(\mathsf{X}^{\mathbb{N}_0}\times \mathsf{Y}^{\mathbb{N}}, \mathcal{B}(\mathsf{X}^{\mathbb{N}_0}\times\mathsf{Y}^{\mathbb{N}})\big)$ that has finite dimensional distributions
	\begin{multline*}
		\eta_{\mathsf{J}_1, \mathsf{J}_2}(\mathsf{A}) := \int_{\mathsf{X}} \int_{(\mathsf{X} \times \mathsf{Y})^{j_m \vee k_\ell}}
		\mathbbm{1}_{\mathsf{A}}(x_{j_1}, \ldots, x_{j_m}, y_{k_1}, \ldots, y_{k_\ell})\\
		 \times (H_1 \otimes H_2)^{\otimes(j_m\vee k_\ell)}\big(x_0, \d (y_1,x_1, \ldots, y_{j_m\vee k_\ell}, x_{j_m\vee k_\ell})\big) \, \mu(\d x_0),
	\end{multline*}
	for $\mathsf{A} \in \mathcal{B}(\mathsf{X}^{ m} \times \mathsf{Y}^{\ell})$, and $\mathsf{J}_1:=\{j_1, \ldots, j_m\}\subseteq \mathbb{N}_0, \mathsf{J}_2:= \{k_1, \ldots, k_\ell\}\subseteq \mathbb{N}$ finite and with elements denoted in increasing order.
	Such a probability measure exists by the Kolmogorov extension theorem.
	For $n \in \mathbb{N}_0$ let $X_n: \mathsf{X}^{\mathbb{N}_0}\times \mathsf{Y}^{\mathbb{N}} \to \mathsf{X}$ be the projection onto the $n$th component of $\mathbb{X}^{\mathbb{N}_0}$.
	Moreover, for $n \in \mathbb{N}$ let $Y_n: \mathsf{X}^{\mathbb{N}_0}\times \mathsf{Y}^{\mathbb{N}} \to \mathsf{Y}$ be the projection onto the $n$th component of $\mathsf{Y}^{\mathbb{N}}$.
	By construction $(X_n)_{n\in \mathbb{N}_0}$ is a Markov chain with transition kernel $K$ and initial distribution $\mu$ under $\mathbb{P}_\mu$.
	Moreover, for all $x \in \mathsf{X}$ and all $n \in \mathbb{N}$ holds $\mathbb{P}_x^{Y_n} = H_1\tilde{K}^{n-1}(x, \cdot)$,
	and we are in the setting of Construction \ref{A: De-initializing construction} with $\Pi$ as specified in the statement.
	
	By virtue of 
	\begin{align*}
		&\mathbb{E}_x\big(\mathbbm{1}_{\mathsf{A}}(X_n) \mathbbm{1}_{\mathsf{B}}(Y_n)\big)
		= (H_1 \otimes H_2)^n(x, \mathsf{B} \times \mathsf{A})\\
		&\qquad =  \int_{\mathsf{Y}} \int_{\mathsf{X}}
			\mathbbm{1}_{\mathsf{A}}(x_n) \mathbbm{1}_{\mathsf{B}}(y_n)\ H_2(y_n, \d x_n)\, (\delta_x H_1)(H_2H_1)^{n-1}(\d y_n)
		= \mathbb{E}_x\big(\mathbbm{1}_{\mathsf{B}}(Y_n) H_2(Y_n, \mathsf{A})\big),
	\end{align*} 
	for $\mathsf{A} \in \mathcal{B}(\mathsf{X})$ and $\mathsf{B} \in \mathcal{B}(\mathsf{Y})$,
	we have for all $x \in \mathsf{X}, n \in \mathbb{N}$ that $\mathbb{P}_x^{X_n \mid Y_n = y} = H_2(y, \cdot)$ for $\mathbb{P}_x^{Y_n}$-almost all $y \in \mathsf{Y}$. 
	Therefore Definition \ref{D: De-initialization} is satisfied with $D_n = H_2$.
	Consequently, if \eqref{Eq: upper Wasserstein-deinitialization condition splitting case} holds, then also Definition \ref{D: Wasserstein de-initialization} is satisfied.

	\textit{Ad b):} 
	First, we provide a suitable Construction \ref{A: De-initializing construction} with reversed roles of $\mathsf{X}$ and $\mathsf{Y}$.
	For $\mu \in \mathcal{M}_{\mathsf{Y}}$ let $\mathbb{P}_\mu$ be the probability measure on $(\mathsf{Y}^{\mathbb{N}_0}\times \mathsf{X}^{\mathbb{N}}, \mathcal{B}(\mathsf{Y}^{\mathbb{N}_0}\times \mathsf{X}^{\mathbb{N}}))$ with finite dimensional distributions
	\[
	\eta_{\mathsf{J}_1,\mathsf{J}_2}(\mathsf{A}) := \begin{dcases}
		\begin{multlined}[b]
			\int_{\mathsf{Y}}\int_{\mathsf{X}} \int_{\mathsf{X}^{j_m \vee k_\ell}} \mathbbm{1}_{\mathsf{A}}\big(y_0, f(x_{j_2}), \ldots, f(x_{j_m}), x_{k_1}, \ldots, x_{k_\ell}\big)\\
			\times	K^{\otimes (j_m \vee k_\ell)}\big(x_0, \d (x_1, \ldots,x_{j_m \vee k_\ell})\big)\,H_2(y_0, \d x_0)\, \mu(\d y_0),
		\end{multlined} & j_1 = 0,\\
		\begin{multlined}[b]
			\int_{\mathsf{Y}} \int_{\mathsf{X}} \int_{\mathsf{X}^{j_m \vee k_\ell}} \mathbbm{1}_{\mathsf{A}}\big(f(x_{j_1}), \ldots, f(x_{j_m}), x_{k_1}, \ldots, x_{k_\ell}\big)\\
			\times	K^{\otimes (j_m \vee k_\ell)}\big(x_0, \d (x_1, \ldots,x_{j_m \vee k_\ell})\big)\,H_2(y_0, \d x_0)\, \mu(\d y_0),
		\end{multlined}  & j_1 \neq 0,
	\end{dcases}
	\]
	for $\mathsf{A} \in \mathcal{B}(\mathsf{X}^{m} \times \mathsf{Y}^{\ell})$, 
	where $\mathsf{J}_1 := \{j_1, \ldots, j_m\} \subseteq \mathbb{N}_0, \mathsf{J}_2:= \{k_1, \ldots, k_\ell\} \subseteq \mathbb{N}$ are finite and their elements are written in increasing order.
	We get the existence of such probability measures from Kolmogorov's extension theorem.
	For $n \in \mathbb{N}_0$ let $Y_n: \mathsf{Y}^{\mathbb{N}_0}\times \mathsf{X}^{\mathbb{N}} \to \mathsf{Y}$ be the projection onto the $n$th component of $\mathsf{Y}^{\mathbb{N}_0}$.
	Moreover, for $n \in \mathbb{N}$ let $X_n: \mathsf{Y}^{\mathbb{N}_0}\times \mathsf{X}^{\mathbb{N}} \to \mathsf{X}$ be the projection onto the $n$th component of $\mathsf{X}^{\mathbb{N}}$.
	
	Observe that for $\mu \in \mathcal{M}_{\mathsf{Y}}$, $n \in \mathbb{N}$, we have
	\begin{align*}
		&\mathbb{P}_\mu^{(Y_0, \ldots, Y_n)}(\mathsf{B})
		 \begin{multlined}[t]
			=\int_{\mathsf{Y}} \int_{\mathsf{X}} \int_{\mathsf{Y}} \int_{\mathsf{X}}  \ldots  \int_{\mathsf{Y}}  \int_{\mathsf{X}}
			\mathbbm{1}_{\mathsf{B}}\big(y_0, f(x_1), \ldots, f(x_n)\big)\\
			\times H_2(y_n, \d x_n)\, H_1(x_{n-1}, \d y_n) \ldots H_2(y_1, \d x_1)\, H_1(x_0, \d y_1)\, H_2(y_0, \d x_0)\, \mu(\d y_0)
		\end{multlined}\\
		&\qquad \overset{\eqref{Eq: Fiber property}}{=} \begin{multlined}[t]
			\int_{\mathsf{Y}} \int_{\mathsf{X}} \int_{\mathsf{Y}} \int_{\mathsf{X}}\ldots  \int_{\mathsf{Y}} \int_{\mathsf{X}}
			\mathbbm{1}_{\mathsf{B}}\big(y_0, y_1, \ldots, y_n\big)\\
			\times H_2(y_n, \d x_n)\, H_1(x_{n-1}, \d y_n) \ldots H_2(y_1, \d x_1)\, H_1(x_0, \d y_1)\, H_2(y_0, \d x_0)\, \mu(\d y_0)
		\end{multlined}\\
		& \qquad = \mu \otimes \tilde{K}^{\otimes n}(\mathsf{B}), \hspace*{0.55\textwidth} \mathsf{B} \in \mathcal{B}(\mathsf{Y}^{n+1}).
	\end{align*}
	Hence, $(Y_n)_{n \in \mathbb{N}_0}$ is a Markov chain with transition kernel $\tilde{K}$ and initial distribution $\mu$ under $\mathbb{P}_\mu$.
	Moreover, $\mathbb{P}_y^{X_n} = H_2K^n(y, \cdot)$ for all $y \in \mathsf{Y}, n \in \mathbb{N}$,
	and we are in the required setting of Construction \ref{A: De-initializing construction} with reversed roles of $\mathsf{X}$ and $\mathsf{Y}$.
	
	Since we have $\mathbb{P}_y^{X_n}$-almost surely $\mathbb{P}_y^{Y_n \mid X_n = \cdot} = \delta_{f(\cdot)}$ for all $y \in \mathsf{Y}, n \in \mathbb{N}$,
	Definition \ref{D: De-initialization} holds with $D_n = \delta_{f(\cdot)}$.
	Moreover, because of \eqref{Eq: wasserstein distance of Dirac measures}, the Lipschitzness of $f$ implies Definition \ref{D: Wasserstein de-initialization}.
\end{proof}

\section{Applications}\label{S: Applications}

In this section we demonstrate different usages of the Wasserstein de-initialization framework.
Firstly, it can be used to fully characterize the Wasserstein convergence behaviour of a Markov chain of interest, for example a MCMC method,
in terms of a simpler Markov chain.
In Sections \ref{S: Linchpin variable sampler} and \ref{S: Gibbs sampler}, we do this for the linchpin variable sampler and the Gibbs sampler, respectively.
Secondly, in Section \ref{S: Slice sampling} we establish Wasserstein ergodicity results for slice sampling.
To obtain these results we crucially exploit that, in contrast to pure de-initialization, 
within the Wasserstein de-initialization framework we can choose the metric on the state space $\mathsf{Y}$ of the de-initializing processes described by $\Pi$.
This leaves the freedom to choose a metric that is convenient to derive good convergence results for the de-initializing processes.
Throughout this section we rely on the Lemmas \ref{L: extended lumping} and \ref{L: De-initialization via splitting} to establish Wasserstein de-initialization.

\subsection{Linchpin variable sampling}\label{S: Linchpin variable sampler}

Suppose we want to generate (approximate) samples from a probability measure $\pi$ having density $p: \mathbb{R}^{d_1+d_2}\to [0, \infty)$ with respect to $d_1+d_2$-dimensional Lebesgue measure $\lambda_{d_1+d_2}$, where $d_1,d_2 \in \mathbb{N}$, in the following setting:
For each value $y \in \mathbb{R}^{d_2}$ of the second block of coordinates, 
the $\mathbb{R}^{d_1}$-valued first block of coordinates has conditional distribution
\begin{equation}\label{Eq: linchpin variable sampler - conditional distribution}
	\pi_{1\vert 2}(y, \d x) = \frac{p(x,y)}{p_L(y)}\ \lambda_{d_1}(\d x), \qquad  y \in \mathbb{R}^{d_2},
\end{equation}
where $p_L(y) := \int_{\mathbb{R}^{d_1}} p(x,y) \, \lambda_{d_1}(\d x)$, $y \in \mathbb{R}^{d_2}$, 
is the density of the marginal distribution $\pi_L$ with respect to $\lambda_{d_2}$ of the second block of coordinates.
By virtue of Tonelli's theorem, we may interpret $\pi_{1 \vert 2}:\mathbb{R}^{d_2} \times \mathcal{B}(\mathbb{R}^{d_1}) \to [0,1]$ as a Markov kernel.
We assume that it is possible to sample directly from each conditional distribution $\pi_{1\vert 2}(y, \cdot)$.
An example for a distribution fitting these assumptions is the Rosenbrock distribution (see Rosenbrock example in the introduction).
If we could additionally sample directly from the marginal distribution $\pi_L$ of the second block of coordinates, which we also call the linchpin variable,
we could assemble the building blocks conditional and marginal distribution to sample directly from our original target distribution $\pi$.
However, instead we only assume to have access to a Markov kernel $L: \mathbb{R}^{d_2}\times \mathcal{B}(\mathbb{R}^{d_2})\to [0,1]$ with invariant distribution $\pi_L$.
Based on this, the \emph{linchpin variable sampler} is constructed by running a Markov chain with transition kernel $L$ on $\mathbb{R}^{d_2}$ 
and then extending each sample to a sample on the full space $\mathbb{R}^{d_1 + d_2}$ by generating the missing first block of coordinates from $\pi_{1\vert 2}$.
This is analogous to the construction of the linchpin variable sampler in the Rosenbrock example.
The corresponding Markov kernel is $\mathtt{LVS} := L \otimes \pi_{1\vert 2}$ on $\mathbb{R}^{d_1+d_2}$.
It has invariant distribution $\pi$, see \cite{vats2022understanding},
and is therefore in principle suited as MCMC method for targeting $\pi$.
See also Figure \ref{F: Linchpin variable sampler} for a graphical illustration of the transition principle.
For a more detailed overview of linchpin variable samplers we refer to \cite{acosta2014Markov,vats2022understanding}.

Within these references it is also observed that the total variation distance behaviour of the linchpin variable sampler is equal to the behaviour of the kernel $L$, 
more precisely, $d_{\mathrm{tv}}\big(L^n(y, \cdot), \pi_L\big)= d_{\mathrm{tv}} \big(\mathtt{LVS}^n((x,y), \cdot), \pi\big)$ for all $(x,y) \in \mathbb{R}^{d_1+d_2}$, $n \in \mathbb{N}$.
The following proposition generalizes this result to Wasserstein distances.

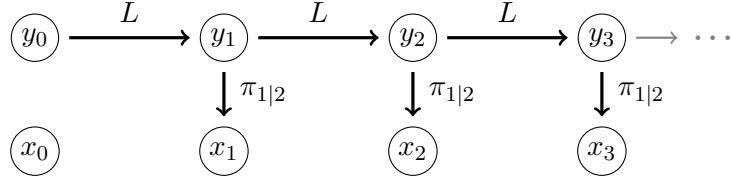
\begin{figure}
	\centering
	\begin{tikzpicture}[inner sep=2pt, outer sep=4pt]
		\coordinate (ynext) at (8.5,0);
		
		\draw (0,0) node[circle,draw] (y0) {$y_0$};
		\draw (0,-1.5) node[circle,draw] (x0) {$x_0$};
		\draw (2.5,0) node[circle,draw] (y1) {$y_1$};
		\draw (2.5,-1.5) node[circle,draw] (x1) {$x_1$};
		\draw (5,0) node[circle,draw] (y2) {$y_2$};
		\draw (5,-1.5) node[circle,draw] (x2) {$x_2$};
		\draw (7.5,0) node[circle,draw] (y3) {$y_3$};
		\draw (7.5,-1.5) node[circle,draw] (x3) {$x_3$};
		
		\path[->, very thick] (y0) edge node[above] {$L$}  (y1) ;
		\path[->, very thick] (y1) edge node[above] {$L$}(y2);
		\path[->, very thick] (y2) edge node[above] {$L$}(y3);
		\path[->,  thick, gray] (y3) edge (ynext);
		\path[->,very thick] (y1) edge node[right] {$\pi_{1 \vert 2}$}  (x1);
		\path[->, very thick] (y2) edge node[right] {$\pi_{1 \vert 2}$} (x2);
		\path[->, very thick] (y3) edge node[right] {$\pi_{1 \vert 2}$} (x3);
		
		\draw (ynext) node[right] {$\color{gray}\boldsymbol{\ldots}$};
	\end{tikzpicture}
	\caption{Transition principle of the linchpin variable sampler.}\label{F: Linchpin variable sampler}
\end{figure}

\begin{proposition}
	In the setting of this section,
	assume that there exists $q \in [1, \infty)$ and $C \in [0, \infty)$ such that
	\begin{equation}\label{Eq: linchpin variable sampler upper Wasserstein de-initialization condition}
		W_{\|\cdot\|_{d_1}}^{(q)}\big(\pi_{1\vert 2}(y_1, \cdot),\pi_{1\vert 2}(y_2 , \cdot)\big)\leq C\cdot \|y_1 - y_2\|_{d_2}, \qquad \forall y_1,y_2 \in \mathbb{R}^{d_2}.
	\end{equation}
	Then for the linchpin variable sampler $\mathtt{LVS}$ holds that for all $(x,y) \in \mathbb{R}^{d_1+d_2}$, $n \in \mathbb{N}$
	\[
		W_{\|\cdot\|_{d_2}}^{(q)}\big(L^n(y, \cdot), \pi_L\big)
		\leq W_{\|\cdot\|_{d_1+d_2}}^{(q)} \big(\mathtt{LVS}^n((x,y), \cdot), \pi\big)
		\leq (C+1) \cdot W_{\|\cdot\|_{d_2}}^{(q)}\big(L^n(y, \cdot), \pi_L\big).
	\]
	In fact, the first inequality holds always for any $q \in [1, \infty)$ without imposing \eqref{Eq: linchpin variable sampler upper Wasserstein de-initialization condition}.
\end{proposition}

\begin{proof}
	We apply Lemma \ref{L: De-initialization via splitting} with $\mathsf{X} = \mathbb{R}^{d_1+d_2}$, $\mathsf{Y} = \mathbb{R}^{d_2}$, $K = \mathtt{LVS}$,
	and $H_1((x,y), \mathsf{A}) = L(y, \mathsf{A})$ for $(x,y) \in \mathbb{R}^{d_1+ d_2}, \mathsf{A} \in \mathcal{B}(\mathbb{R}^{d_2})$ and
	\[
		H_2(y, \mathsf{B}) = \int_{\mathbb{R}^{d_1}} \mathbbm{1}_{\mathsf{B}}(\tilde{x}, y)\ \pi_{1 \vert 2}(y, \d \tilde{x}),
		\qquad y \in \mathbb{R}^{d_2}, \mathsf{B} \in \mathcal{B}(\mathbb{R}^{d_1+d_2}).
	\]
	Then $\tilde{K} = H_2 H_1 = L$,
	and the kernel $\mathtt{LVS}$ is de-initialized by the mapping given by $\Pi_L(\mu,n) = \mu_{\mathbb{R}^{d_2}} L^n$ for $\mu \in \mathcal{M}_{\mathbb{R}^{d_1+d_2}}$, $n \in \mathbb{N}$.
	
	If we apply Lemma \ref{L: Tensor product Wasserstein distance} with $\mathsf{A} = \mathsf{B} = \mathbb{R}^{d_2}$, $\mathsf{C}=\mathbb{R}^{d_1}$ equipped with their respective Euclidean metrics, $S(y,\cdot) = \delta_y$ for $y \in \mathbb{R}^{d_2}$, $T = \pi_{1 \vert 2}$, $K_1=1$ and $K_2 = C$ (from \eqref{Eq: linchpin variable sampler upper Wasserstein de-initialization condition}),
	we get 
	\[
		W_{\|\cdot\|_{d_1 + d_2}}^{(q)}\big(H_2(y_1, \cdot), H_2(y_2, \cdot)\big)\leq (C+1)\cdot  \|y_1 - y_2\|_{d_2}, \qquad \forall y_1,y_2 \in \mathbb{R}^{d_2}.
	\]
	Hence, $\Pi_L$ is also $q$-Wasserstein de-initializing for $\mathtt{LVS}$ with constant $C+1$.

	Choose the function $f$ in Lemma \ref{L: De-initialization via splitting} as the projection $f: \mathbb{R}^{d_1 + d_2} \to \mathbb{R}^{d_2}, (x,y) \mapsto y$.
	Then \eqref{Eq: Fiber property} holds by construction, 
	such that the mapping given by $\Pi_{\mathtt{LVS}}(\mu,n) = (\mu \otimes \pi_{1 \vert 2})\mathtt{LVS}^n$ for $\mu \in \mathcal{M}_{\mathbb{R}^{d_2}}$, $n \in \mathbb{N}$ is de-initializing for $\mathtt{L}$.
	Since $f$ is 1-Lipschitz continuous, $\Pi_{\mathtt{LVS}}$ is also $q$-Wasserstein de-initializing for $L$ with constant 1.
	
	Finally, the Wasserstein distance estimates follow from Theorem \ref{Thm: Wasserstein de-initialization} by virtue of $\pi_L \otimes \pi_{1 \vert 2} = \pi$, and the invariance of $\mathtt{LVS}$ and $L$ with respect to $\pi$ and $\pi_L$, respectively.  
\end{proof}

\subsection{Gibbs sampling}\label{S: Gibbs sampler}

Consider a Polish space $(\mathsf{E}, \mathcal{B}(\mathsf{E}))$ that is the product of two Polish spaces $\mathsf{E}_1$ and $\mathsf{E}_2$,
and a reference measure $\Lambda:= \Lambda_1 \otimes \Lambda_2$ on $(\mathsf{E}, \mathcal{B}(\mathsf{E}))$ that is the product measure of two $\sigma$-finite measures $\Lambda_1$ and $\Lambda_2$ on $\mathsf{E}_1$ and $\mathsf{E}_2$, respectively.
Say we are interested in generating samples from a distribution $\pi(\d(x,y)) := p(x,y)\Lambda(\d(x,y))$ on $\mathsf{E}$,  
and we are able to sample directly from the conditional distributions
\[
	\pi_{1 \vert 2}(y, \d \tilde{x}) = \frac{p(\tilde{x},y)}{p_2(y)}\Lambda_1(\d \tilde{x})
	\quad \text{and} \quad
	\pi_{2 \vert 1}(x, \d \tilde{y}) = \frac{p(x,\tilde{y})}{p_1(x)}\Lambda_2(\d \tilde{y})
	\quad \text{for all } x\in \mathsf{E}_1, y\in \mathsf{E}_2,
\]
where $p_1(x):= \int_{\mathsf{E}_2} p(x,\tilde{y})\Lambda_2(\d \tilde{y})$, $x \in \mathsf{E}_1$, is the density of the $\mathsf{E}_1$-valued marginal distribution $\pi_1(\d x):= p_1(x)\Lambda_1(\d x)$,  and $p_2(y):= \int_{\mathsf{E}_1} p(\tilde{x},y)\Lambda_1(\d \tilde{x})$, $y \in \mathsf{E}_2$, is the density of the $\mathsf{E}_2$-valued marginal distribution $\pi_2(\d y):= p_2(y)\Lambda_2(\d y)$ of $\pi$.
Note that by virtue of Tonelli's theorem  we may interpret $\pi_{2 \vert 1}: \mathsf{E}_1 \times \mathcal{B}(\mathsf{E}_2) \to [0,1]$ and $\pi_{1 \vert 2}: \mathsf{E}_2 \times \mathcal{B}(\mathsf{E}_1) \to [0,1]$ as Markov kernels.
In this setting, Gibbs sampling is a popular MCMC method to generate approximate samples from $\pi$, see e.g.\ \cite{gelman2014Bayesian, robert2004Monte}.
In this article, we focus on 2-stage deterministic scan Gibbs sampling, 
i.e.\ we alternately update the $\mathsf{E}_1$-valued and $\mathsf{E}_2$-valued block of coordinates by sampling according to $\pi_{1 \vert 2}$ and $\pi_{2 \vert 1}$, respectively.
We denote the resulting transition kernel as
\[
	\mathtt{Gibbs}\big((x,y), \mathsf{A}\big) :=  \int_{\mathsf{E}_1} \int_{\mathsf{E}_2} \mathbbm{1}_{\mathsf{A}}\big((\tilde{x}, \tilde{y})\big)\ \pi_{2 \vert 1}(\tilde{x}, \d \tilde{y})\, \pi_{1 \vert 2}(y, \d \tilde{x}), \qquad (x,y) \in \mathsf{E}, \mathsf{A} \in \mathcal{B}(\mathsf{E}).
\]
See also Figure \ref{F: Gibbs sampling} for an illustration of the transition mechanism of Gibbs sampling.

\begin{figure}
	\centering
	\begin{tikzpicture}[inner sep=2pt, outer sep=4pt]
		\draw (0,0) node[circle,draw] (x0) {$x_0$};
		\draw (0,-1.75) node[circle,draw] (y0) {$y_0$};
		\draw (2.5,0) node[circle,draw] (x1) {$x_1$};
		\draw (2.5,-1.75) node[circle,draw] (y1) {$y_1$};
		\draw (5,0) node[circle,draw] (x2) {$x_2$};
		\draw (5,-1.75) node[circle,draw] (y2) {$y_2$};
		\draw (7.5,0) node[circle,draw] (x3) {$x_3$};
		\draw (7.5,-1.75) node[circle,draw] (y3) {$y_3$};
		
		\draw (10,0) node[circle] (ynext) {$\color{gray}\boldsymbol{\ldots}$};
		
		\path[->, very thick] (y0) edge node[midway, xshift=-7pt, yshift=7pt] {$\pi_{1 \vert 2}$}  (x1) ;
		\path[->, very thick] (x1) edge node[left, xshift=5pt] {$\pi_{2 \vert 1}$}(y1);
		\path[->,  thick, gray] (y3) edge (ynext);
		\path[->, very thick] (y1) edge node[midway, xshift=-7pt, yshift=7pt] {$\pi_{1 \vert 2}$} (x2);
		\path[->,very thick] (x2) edge node[left, xshift=5pt] {$\pi_{2 \vert 1}$}  (y2);
		\path[->, very thick] (y2) edge node[midway, xshift=-7pt, yshift=7pt] {$\pi_{1 \vert 2}$} (x3);
		\path[->,very thick] (x3) edge node[left, xshift=5pt] {$\pi_{2 \vert 1}$}  (y3);
		
	\end{tikzpicture}
	\caption{Transition principle of Gibbs sampling.}\label{F: Gibbs sampling}
\end{figure}
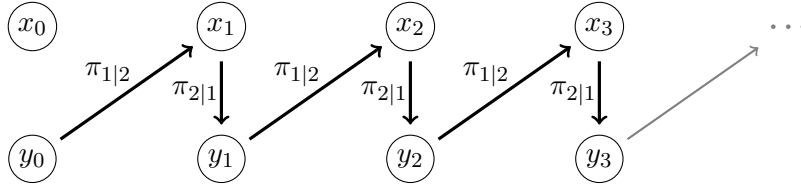

If we only consider the first block of coordinates, we obtain again a Markov chain, which is $\mathsf{E}_1$-valued and has transition kernel $\mathtt{Gibbs}_1:=\pi_{2 \vert 1}\pi_{1 \vert 2}$.
Similarly, the second block of coordinates is a $\mathsf{E}_2$-valued Markov chain with transition kernel $\mathtt{Gibbs}_2:= \pi_{1\vert 2} \pi_{2 \vert 1}$,
and we can also flip the order in which we update the two blocks of coordinates resulting in a ``flipped'' Gibbs sampler with transition kernel
\[
	\mathtt{FGibbs}\big((x,y), \mathsf{A}\big) :=  \int_{\mathsf{E}_2} \int_{\mathsf{E}_1} \mathbbm{1}_{\mathsf{A}}\big((\tilde{x}, \tilde{y})\big)\ \pi_{1 \vert 2}(\tilde{y}, \d \tilde{x})\, \pi_{2 \vert 1}(x, \d \tilde{y}), \qquad (x,y) \in \mathsf{E}, \mathsf{A} \in \mathcal{B}(\mathsf{E}).
\]
\cite{roberts2001markov} observe that in terms of total variation distance, 
the convergence behaviour of the Gibbs sampler is controlled by the behaviour of $\mathtt{Gibbs}_2$, and similarly also by the behaviour of $\mathtt{Gibbs}_1$.
Moreover, the convergence behaviour of $\mathtt{Gibbs}$ and $\mathtt{Gibbs}_2$ is also known to be related in terms of  weak Poincar{\'e} inequalities \cite[Theorems 3.4 and 3.5]{gao2026weak}.
The following proposition extends these results to Wasserstein distance and covers also the flipped Gibbs sampler.

\begin{proposition}\label{P: De-initialization for Gibbs sampling}
	In the setting of this section,
	let $\mathrm{d}_1: \mathsf{E}_1 \times \mathsf{E}_1 \to [0, \infty)$ and $\mathrm{d_2}:\mathsf{E}_2 \times \mathsf{E}_2 \to [0, \infty)$ be two lower semi-continuous metrics. Moreover, let $\mathrm{d}:\mathsf{E} \times \mathsf{E} \to [0, \infty)$ be the $\mathfrak{p}$-product metric of $\mathrm{d}_1$ and $\mathrm{d}_2$ on $\mathsf{E}$ for some $\mathfrak{p} \in [1, \infty]$.
	\begin{enumerate}[label=\alph*)]
		\item Assume there exists $q \in [1, \infty)$ and $C_1 \in [0, \infty)$ such that
			\begin{equation}\label{Eq: condition Wasserstein de-initialiaztation with x-gibbs sampling upper bound}
				W^{(q)}_{\mathrm{d}_2}\big(\pi_{2 \vert 1}(x_1, \cdot), \pi_{2 \vert 1}(x_2, \cdot)\big) \leq C_1\cdot \mathrm{d}_1(x_1,x_2), \qquad \forall x_1,x_2 \in \mathsf{E}_1.
			\end{equation}
			Then we have for all $(x,y) \in \mathsf{E}$, $n \in \mathbb{N}$ that
			\begin{multline*}
				W^{(q)}_{\mathrm{d}_1}\big(\pi_{1 \vert 2}\mathtt{Gibbs}_1^{n-1}(y,\cdot), \pi_1\big)
				\leq W^{(q)}_{\mathrm{d}}\Big(\mathtt{Gibbs}^n\big((x,y), \cdot\big), \pi\Big)\\
				\leq (C_1 + 1)\cdot W^{(q)}_{\mathrm{d}_1}\big(\pi_{1 \vert 2}\mathtt{Gibbs}_1^{n-1}(y,\cdot), \pi_1\big),
			\end{multline*}
			In fact, the first inequality in the previous equation holds always for any $q \in [1, \infty)$ also without \eqref{Eq: condition Wasserstein de-initialiaztation with x-gibbs sampling upper bound}. Moreover, we have for all $(x,y \in \mathsf{E})$, $n \in \mathbb{N}$
			\[
				W^{(q)}_{\mathrm{d}}\Big(\mathtt{Gibbs}^n\big((x,y), \cdot\big), \pi\Big)
				\leq (C_1+1)\cdot W^{(q)}_{\mathrm{d}}\big(\pi_{1 \vert 2}\mathtt{FGibbs}^{n-1}(y, \cdot), \pi\big).
			\]
		\item If there exists $q \in [1, \infty)$ and $C_2 \in [0, \infty)$ such that for all $(x_1,y_1),(x_2,y_2) \in \mathsf{E}$ holds
			\begin{equation}\label{Eq: condition Wasserstein de-initialiaztation with y-gibbs sampling upper bound}
				W^{(q)}_{\mathrm{d}}\Big(\mathtt{Gibbs}\big((x_1,y_1), \cdot\big), \mathtt{Gibbs}\big((x_1,y_1), \cdot\big)\Big) 
				\leq C_2\cdot \mathrm{d}_2(y_1,y_2),
			\end{equation}
			then we have for all $(x,y) \in \mathsf{E}, n \in \mathbb{N}$ that
			\[
				W^{(q)}_{\mathrm{d}_2}\big(\mathtt{Gibbs}_2^{n}(y,\cdot), \pi_2\big)
				\leq W^{(q)}_{\mathrm{d}}\Big(\mathtt{Gibbs}^n\big((x,y), \cdot\big), \pi\Big)
				\leq C_2\cdot W^{(q)}_{\mathrm{d}_2}\big(\mathtt{Gibbs}_2^{n-1}(y,\cdot), \pi_2\big).
			\]
			In fact, the first inequality holds always for any $q \in [1, \infty)$ without assuming \eqref{Eq: condition Wasserstein de-initialiaztation with y-gibbs sampling upper bound}.
	\end{enumerate}
\end{proposition}
\begin{proof}
	\textit{Ad a):}
	We first tend to the chain of inequalities involving $\mathtt{Gibbs}_1$.
	To this end, we apply Lemma \ref{L: De-initialization via splitting} with $\mathsf{X} = \mathsf{E}$, $\mathsf{Y} = \mathsf{E}_1$, $K = \mathtt{Gibbs}$, and
	$H_1 :\mathsf{E} \times \mathcal{B}(\mathsf{E}_1)\to [0,1], ((x,y), \mathsf{A})\mapsto \pi_{1 \vert 2}(y, \mathsf{A})$, and
	$H_2 : \mathsf{E}_1 \times \mathcal{B}(\mathsf{E})\to [0,1], (x, \mathsf{C}) \mapsto \int_{\mathsf{E}_2} \mathbbm{1}_{\mathsf{C}}((x,y))\ \pi_{2 \vert 1}(x, \d y)$.
	Then $\tilde{K}= H_2H_1 = \pi_{2 \vert 1}\pi_{1 \vert 2}=\mathtt{Gibbs}_1$.
	We get that $\mathtt{Gibbs}$ is de-initialized by the mapping given by $\Pi_1(\mu, n) = \mu_{\mathsf{E}_2}\pi_{1 \vert 2}\mathtt{Gibbs}_1^{n-1}$ for $\mu \in \mathcal{M}_{\mathsf{E}}$, $n \in \mathbb{N}$.
	Using \eqref{Eq: condition Wasserstein de-initialiaztation with x-gibbs sampling upper bound},
	we get from an application of Lemma \ref{L: Tensor product Wasserstein distance} with $\mathsf{A} = \mathsf{B}=\mathsf{E}_1$, $\mathsf{C} = \mathsf{E}_2$ equipped with $\mathrm{d}_1$ and $\mathrm{d}_2$, respectively, $S(x, \cdot)= \delta_x$ for $x \in \mathsf{E}_1$, $T= \pi_{2 \vert 1}$, $K_1 = 1$ and $K_2 = C_1$, that
	\[
		W_{\mathrm{d}}^{(q)}\big(H_2(x_1, \cdot), H_2(x_2,\cdot )\big) \leq (C_1 + 1) \cdot \mathrm{d}_{1}(x_1,x_2), \qquad \forall x_1,x_2 \in \mathsf{E}_1.
	\]
	Hence, $\Pi_1$ is also $q$-Wasserstein de-initializing for $\mathtt{Gibbs}$.
	
	If we choose $f: \mathsf{E} \to \mathsf{E}_1, (x,y) \mapsto x$ as the projection onto $\mathsf{E}_1$ in Lemma \ref{L: De-initialization via splitting}, 
	then \eqref{Eq: Fiber property} holds by construction.
	Therefore the mapping given by $\Pi_{\mathtt{G}1}(\mu,n) = (\mu \otimes \pi_{2 \vert 1})\mathtt{Gibbs}^n$ for $\mu \in \mathcal{M}_{\mathsf{E}_1}$, $n \in \mathbb{N}$ is de-initializing for $\mathtt{Gibbs}_1$. 
	As $f$ is 1-Lipschitz, $\Pi_{\mathtt{G}1}$ is also $q$-Wasserstein de-initializing for $\mathtt{Gibbs}_1$ with constant 1.
	
	The Wasserstein inequalities follow from Theorem \ref{Thm: Wasserstein de-initialization} using the invariance of $\mathtt{Gibbs}$ and $\mathtt{Gibbs}_1$ with respect to $\pi$ and $\pi_1$, respectively, and $(\pi_{1 \vert 2}(y, \cdot) )\otimes \pi_{2 \vert 1} = \mathtt{Gibbs}((x,y), \cdot)$ for all $(x,y) \in \mathsf{E}$.
	
	We turn to the statement about the flipped Gibbs sampler.
	We apply Lemma \ref{L: De-initialization via splitting}.a) for $\mathsf{X} = \mathsf{Y} = \mathsf{E}$, $K=\mathtt{Gibbs}$, and 
	$H_1:\mathsf{E}\times \mathcal{B}(\mathsf{E}), ((x,y), \mathsf{C})\mapsto \int_{\mathsf{X}} \mathbbm{1}_{\mathsf{C}}(\widetilde{x}, y)\ \pi_{1 \vert 2}(y, \d \widetilde{x})$
	and
	$H_2: \mathsf{E}\times \mathcal{B}(\mathsf{E}), 
	 ((x,y), \mathsf{C})\mapsto \int_{\mathsf{Y}} \mathbbm{1}_{\mathsf{C}}(x, \widetilde{y})\ \pi_{2 \vert 1}(x, \d \widetilde{y})$.
	Consequently $\tilde{K}=H_2H_1 = \mathtt{FGibbs}$.
	We get that the mapping given by $\Pi_{\mathtt{F}}(\mu, n) = (\mu_{\mathsf{E}_2}\otimes \pi_{1 \vert 2})\mathtt{FGibbs}^{n-1}$ is de-initializing for $\mathtt{Gibbs}$.
	An application of Lemma \ref{L: Tensor product Wasserstein distance} with $\mathsf{A}=\mathsf{E}$, $\mathsf{B}=\mathsf{E}_1$, $\mathsf{C}=\mathsf{E}_2$,
	$S((x,y), \cdot) = \delta_x$ for $x \in \mathsf{E}_1$, $T = \pi_{2 \vert 1}$, $K_1=1$ and $K_2 = C_1$ (by \eqref{Eq: condition Wasserstein de-initialiaztation with x-gibbs sampling upper bound}) yields
	\[
		W_{\mathrm{d}}^{(q)}\big(H_2(z_1, \cdot), H_2(z_2, \cdot)\big) \leq (C_1 + 1) \cdot \mathrm{d}(z_1,z_2), \qquad \forall z_1,z_2\in \mathsf{E}.
	\]
	Hence, $\Pi_{\mathtt{F}}$ is also $q$-Wasserstein de-initializing for $\mathtt{Gibbs}$.
	The Wasserstein inequality follows then by Theorem \ref{Thm: Wasserstein de-initialization} and the invariance of $\mathtt{Gibbs}$ and $\mathtt{FGibbs}$ with respect to $\pi$.
	
	\textit{Ad b):}
	We apply Lemma \ref{L: extended lumping} with $\mathsf{X} = \mathsf{E}$,
	$\mathsf{Y} = \mathsf{E}_2$, $K = \mathtt{Gibbs}$ and $f: \mathsf{E}\to \mathsf{E}_2, (x,y) \mapsto y$.
	We get $R = \pi_{1 \vert 2} \otimes \pi_{2 \vert 1}$ and $Q = \pi_{1 \vert 2}\pi_{2 \vert 1}=\mathtt{Gibbs}_2$,
	such that the mapping given by $\Pi_2(\mu,n) = \mu_{\mathsf{E}_2}\mathtt{Gibbs}_2^{n-1}$ for $\mu \in \mathcal{M}_{\mathsf{E}}$, $n \in \mathbb{N}$ is de-initializing for $\mathtt{Gibbs}$.
	Since $\mathtt{Gibbs}((x,y), \cdot) =  (\pi_{1 \vert 2} \otimes \pi_{2 \vert 1})(y, \cdot)=R(y, \cdot)$ for all $(x,y) \in \mathsf{E}$,
	Condition \eqref{Eq: condition Wasserstein de-initialiaztation with y-gibbs sampling upper bound} yields
	\[
		W^{(q)}_{\mathrm{d}}\big(R(y_1, \cdot), R(y_2, \cdot)\big) \leq C_2 \cdot \mathrm{d}_2(y_1,y_2), \qquad \forall y_1,y_2 \in \mathsf{E}_2.
	\]
 	Hence $\Pi_2$ is also $q$-Wasserstein de-initializing for $\mathtt{Gibbs}$.
	
	On the other hand, because $f$ is 1-Lipschitz, the mapping given by $\Pi_{\mathtt{G}2}(\mu,n) = \mu R \mathtt{Gibbs}^{n-1}$ for $\mu \in \mathcal{M}_{\mathsf{E}_2}$, $n \in \mathbb{N}$ is $q$-Wasserstein de-initializing for $\mathtt{Gibbs}_2$ with constant 1.
	Overall, the Wasserstein inequalities follow then from Theorem \ref{Thm: Wasserstein de-initialization} using $\mu_{\mathsf{E}_2}R=\mu\mathtt{Gibbs}$ for $\mu \in \mathcal{M}_{\mathsf{E}}$ and the invariance of $\mathtt{Gibbs}$ and $\mathtt{Gibbs}_2$ with respect to $\pi$ and $\pi_2$, respectively.
\end{proof}
The following special case may be helpful to establish \eqref{Eq: condition Wasserstein de-initialiaztation with y-gibbs sampling upper bound}.
\begin{remark}
	If there exists  $\widetilde{C}_1,\widetilde{C}_2 \in [0, \infty)$ and $q \in [1, \infty)$ such that
	\begin{align*}
		W_{\mathrm{d}_1}^{(q)}\big(\pi_{1 \vert 2}(y_1, \cdot), \pi_{1 \vert 2}(y_2, \cdot)\big) &\leq \widetilde{C}_1 \cdot\mathrm{d}_2(y_1,y_2),\qquad \forall y_1,y_2 \in \mathsf{E}_2, \text{and}\\
		W_{\mathrm{d}_2}^{(q)}\big(\pi_{2 \vert 1}(x_1, \cdot), \pi_{2 \vert 1}(x_2, \cdot)\big) &\leq \widetilde{C}_2 \cdot\mathrm{d}_1(x_1,x_2), \qquad \forall x_1,x_2 \in \mathsf{E}_1,
	\end{align*}
	then \eqref{Eq: condition Wasserstein de-initialiaztation with y-gibbs sampling upper bound} holds with $C_2 = \widetilde{C}_1(\widetilde{C}_2+1)$.
	This follows from an application of Lemma \ref{L: Tensor product Wasserstein distance} with $\mathsf{A}=\mathsf{C} = \mathsf{E}_2$, $\mathsf{B} = \mathsf{E}_1$, $S= \pi_{1 \vert 2}$ and $T=\pi_{2 \vert 1}$.
\end{remark}

The following corollary illustrates the implications of Proposition \ref{P: De-initialization for Gibbs sampling} by using it to leverage  the special convergence properties of Markov chains on finite state spaces for Gibbs sampling.
In terms of total variation distance this setting is considered by \cite{diebolt1994estimation}, \cite[End of Section 3.3]{roberts2001markov}.

\begin{corollary}
	In the setting of this section, if $p>0$ and either $\mathsf{E}_1$ or $\mathsf{E}_2$ are finite, then $\mathtt{Gibbs}$ is uniformly Wasserstein ergodic for all $q \in [1, \infty)$,
	i.e.\ for all $q \in [1, \infty)$ there exist $C \in [0, \infty)$ and $\varrho \in [0,1)$ such that
	\[
		W^{(q)}_{\mathrm{d}}\Big(\mathtt{Gibbs}^n\big((x,y), \cdot\big), \pi\Big) \leq C \varrho^n, \qquad \forall (x,y) \in \mathsf{E}, n \in \mathbb{N}.
	\]
\end{corollary}
\begin{proof}
	Fix $q \in [1, \infty)$.
	First assume that $\mathsf{E}_1$ is finite.
	Then \eqref{Eq: condition Wasserstein de-initialiaztation with x-gibbs sampling upper bound} is trivially satisfied.
	The claim follows from Proposition \ref{P: De-initialization for Gibbs sampling}.a), because $\mathtt{Gibbs}_{1}$ induces irreducible, aperiodic Markov chains on a finite state space,
	which are always uniformly Wasserstein ergodic by virtue of \cite[Theorem 6.15]{villani2009optimal} and their uniform geometric ergodicity with respect to total variation distance.
	
	If $\mathsf{E}_2$ is finite, the statement follows by a similar argument using Proposition \ref{P: De-initialization for Gibbs sampling}.b) and $\mathtt{Gibbs}((x,y), \cdot) =  (\pi_{1 \vert 2} \otimes \pi_{2 \vert 1})(y, \cdot)$ for all $(x,y) \in \mathsf{E}$.
\end{proof}

We finish this section with an explicit example of an application of Proposition \ref{P: De-initialization for Gibbs sampling}.
\begin{example}
	Let $\alpha \in (0, 1)$. 
	Consider the Gibbs sampler on $\mathbb{R}^2$ build from $\pi_{1\vert2}(z, \cdot) = \pi_{2 \vert 1}(z, \cdot) = \mathcal{N}(\alpha z, 1- \alpha^2)$, $z \in \mathbb{R}$,
	where $\mathcal{N}(\mu, \sigma^2)$ denotes the 1-dimensional normal distribution with mean $\mu$ and variance $\sigma^2$.
	This Gibbs sampler is invariant with respect to the 2-dimensional normal distribution with mean $\left(\begin{smallmatrix}
		0\\
		0\end{smallmatrix}\right)$ and covariance matrix $\left(\begin{smallmatrix}
																											1 & \alpha\\
																											\alpha & 1\end{smallmatrix}\right)$,
	and for the transition kernel of the marginal chain of the first coordinate holds $\mathtt{Gibbs}_1(x, \cdot) = \mathcal{N}(\alpha^2 x, (1-\alpha^4)^2)$,
	see \cite[Example 9.1]{robert2004Monte}.
	Using the explicit formula for the 2-Wasserstein distance of normal distributions,
	we get $W_{|\cdot|}^{(2)}(\mathtt{Gibbs}_1(x_1, \cdot), \mathtt{Gibbs}_1(x_1, \cdot)) = \alpha^2| x_1 - x_2|$ for all $x_1,x_2 \in \mathbb{R}$,
	such that $\mathtt{Gibbs}_1$ is Wasserstein contractive with rate $\alpha^2$.
	We can leverage this convergence behaviour to the whole Gibbs sampler  by applying Proposition \ref{P: De-initialization for Gibbs sampling}.a):
	Since $W_{|\cdot|}^{(2)}(\pi_{2 \vert 1}(x_1, \cdot), \pi_{2 \vert 1}(x_2, \cdot)) = \alpha |x_1-x_2|$ for all $x_1,x_2 \in \mathbb{R}$, 
	we have
	\[
		W_{\|\cdot\|}^{(2)}\left (\mathtt{Gibbs}^n(z, \cdot), \mathcal{N}\big(\begin{pmatrix}
																		0\\ 0
																		\end{pmatrix},
																		\begin{pmatrix}
																			1 & \alpha\\
																			\alpha & 1\end{pmatrix}\big)\right )
				\leq \frac{1+\alpha}{\alpha^2} \cdot \alpha^{2n}, \qquad \forall z \in \mathbb{R}^2, n \in \mathbb{N}.
	\]
\end{example}

\subsection{Slice sampling}\label{S: Slice sampling}
In this section we assume that our distribution of interest $\pi$ is of the following form.
\begin{assumption}\label{A: Slice sampling}
	Let $(\mathsf{X}, \mathcal{B}(\mathsf{X}))$ be a Polish space equipped with a $\sigma$-finite reference measure $\varrho$.
	Consider $\pi(\d x) \propto p(x)\varrho(\d x)$,
	where $p: \mathsf{X} \to (0, \infty)$ is a measurable function with $\int_{\mathsf{X}} p\, \d \varrho \in (0,\infty)$,
	$\sup_{x \in \mathsf{X}}|p(x)| = \sup_{x \in \mathsf{X}\setminus \mathsf{N}}|p(x)|$ for all $\mathsf{N} \in \mathcal{B}(\mathsf{X})$ with $\varrho(\mathsf{N})= 0$, and
	$p(\mathsf{X})\subseteq (0,\infty)$ is an interval.
\end{assumption}
For the unnormalized density $p$ we define its \emph{(super) level sets} $\mathsf{L}(t):= \{x \in \mathsf{X} \mid p(x)> t\}$, $t \in [0, \infty)$, and its \emph{level set function} $\ell: (0, \infty) \to [0, \infty), t \mapsto \varrho\big(\mathsf{L}(t)\big)$.
With these objects the kernel of the \emph{ideal slice sampler} with invariant distribution $\pi$ is 
\[
	\mathtt{S}: \mathsf{X} \times \mathcal{B}(\mathsf{X}) \to [0,1], \qquad
	(x, \mathsf{A}) \mapsto \frac{1}{p(x)}\int_{(0, p(x))} \frac{1}{\ell(s)} \int_{\mathsf{L}(s)}\mathbbm{1}_{\mathsf{A}}(y)\ \varrho(\d y)\, \lambda_{1}(\d s).
\] 
In words, the transition mechanism of ideal slice sampling first samples a new level $t \in (0, \infty)$ below the current value of the unnormalized target density $p$.
The next state is then sampled from the normalized restriction of the reference measure $\varrho$ to the level set $\mathsf{L}(t)$.
For more details, we refer to \cite{schaer2025slice, hasenpflug2024slice}.

Let $\mathrm{d}:\mathsf{X}\times \mathsf{X}\to[0, \infty)$ be some lower semi-continuous pseudo metric.
Our aim is to establish conditions for Wasserstein ergodicity of  ideal slice sampling in terms of the 1-Wasserstein distance induced by $\mathrm{d}$.
To formulate these conditions, we define for $\alpha \in [0, \infty), \kappa \in (0, \infty)$ the functions
\begin{equation}\label{Eq: definition m_alpha_kappa}
	m_{\alpha, \kappa}: \mathsf{W}^2 \to [0,\infty), \qquad (r,t) \mapsto \begin{dcases}
																					\left|(\ell(r)/\kappa)^\alpha -  (\ell(t)/\kappa)^\alpha\right|,& \alpha > 0,\\
																					|\ln(\ell(r)) - \ln(\ell(t))|, & \alpha = 0, \ell> 0 \text{ on } \mathsf{W},
																				\end{dcases}
\end{equation}
where $\mathsf{W}:= p(\mathsf{X}) \cup (0, \inf_{x \in \mathsf{X}}p(x)]$.
To ensure well definedness, we only allow $\alpha = 0$ if the level set function is strictly positive on $\mathsf{W}$.
This is implicitly imposed throughout the section.
If $\kappa = 1$, we also omit the corresponding index.
\begin{proposition}\label{P: Wasserstein ergodicity of slice sampling}
	Impose Assumption \ref{A: Slice sampling}. Let the level set function $\ell: p(\mathsf{X}) \to [0, \infty)$ be continuous and strictly decreasing with $\lim_{t \to \|p\|_{\infty}} \ell(t)= 0$,
	and assume that there exists $\alpha \in [0, \infty)$ such that
	\begin{enumerate}[label=(\roman*)]
		\item for some  $\beta < 1+ \alpha$ holds
		$\int_{(0,1)} m_\alpha(sr, st)\ \lambda_1(\d s) \leq \beta \cdot m_\alpha(r,t)$ for all $r,t \in p(\mathsf{X})$,
		\item for some $C \in [0, \infty)$ and $\kappa \in (0, \infty)$ holds
		$W_{\mathrm{d}}\left(\frac{1}{\ell(r)} \varrho\vert_{\mathsf{L}(r)}, \frac{1}{\ell(t)} \varrho\vert_{\mathsf{L}(t)}\right)\leq C\cdot m_{\alpha, \kappa}(r,t)$ for all $r,t \in \mathsf{W}$.
	\end{enumerate}
	Then we have
	$W_{\mathrm{d}}(\mu\mathtt{S}^n, \pi) \leq C(1+\alpha) \cdot \left(\frac{\beta}{1+\alpha}\right)^n W_{m_{\alpha, \kappa}}\big(p(\mu), p(\pi)\big)$ for all $\mu \in \mathcal{M}_{\mathsf{X}}, n \in \mathbb{N}$.
\end{proposition}

Define for $\alpha \in[0, \infty)$, $\kappa \in (0, \infty)$
\[
g_{\alpha, \kappa} : \begin{dcases}
	[0, \infty) \to [0, \infty), \qquad x \mapsto (x/\kappa)^\alpha, & \alpha > 0\\
	(0, \infty) \to \mathbb{R}, \qquad x \mapsto \ln(x/\kappa), & \alpha = 0.
	\end{dcases}
\]
If $g_{\alpha, \kappa} \circ \ell \circ p:\mathsf{X} \to \mathbb{R}$ is well-defined and $K$-Lipschitz continuous, 
then the previous statement implies a uniform geometric rate of convergence for the slice sampler to its invariant distribution by virtue of Remark \ref{R: How to apply Rueschendorf} (with $\eta= p(\mu)$ and $\xi = p(\pi)$), i.e.
\begin{equation}\label{Eq: Slice sampler Wasserstein ergodicity}
	W_{\mathrm{d}}\big(\mu \mathtt{S}^n, \pi\big) \leq C(1+\alpha) K \cdot \left(\frac{\beta}{1+\alpha}\right)^n W_{\mathrm{d}}(\mu, \pi), \qquad \forall \mu \in \mathcal{M}_{\mathsf{X}}, n\in \mathbb{N}.
\end{equation}

We obtain Proposition \ref{P: Wasserstein ergodicity of slice sampling} by analysing suitable $\mathbb{R}$-valued Markov chains that are Wasserstein de-initializing for the slice sampler under the assumptions of the previous proposition.
Namely inspecting $\mathtt{S}$ more closely, we observe that it is lumping with respect to the measurable, surjective function $p: \mathsf{X}\to p(\mathsf{X})$, see Remark \ref{R: Lumping}.
The corresponding lumped kernel 
\[
	\mathtt{L}: p(\mathsf{X}) \times \mathcal{B} \big(p(\mathsf{X})\big), \qquad (t,\mathsf{B}) \mapsto \frac{1}{t} \int_{(0,t)} \frac{p(\varrho)\big((s , \infty) \cap \mathsf{B}\big)}{\ell(s)} \ \lambda_{1}(\d s)
\]
describes the evolution of the density values during slice sampling.
This is again a slice sampling kernel on the space $(p(\mathsf{X}), \mathcal{B}(p(\mathsf{X})), p(\varrho))$ with the identity function as unnormalized target density.
Since $t\cdot p(\varrho)(\d t) = p(\pi)(\d t)$, this yields that $\mathtt{L}$ has invariant distribution $p(\pi)$.

For two convergence concepts it has been established that Markov kernels related to the evolution of the values of $p$ 
are the decisive factor for the behaviour of slice sampling:
\cite{roberts2001markov} observe\footnote{
	If $(X_n)_{n \in \mathbb{N}}$ is a slice sampler on $(\mathbb{R}^d, \mathcal{B}(\mathbb{R}^d), \lambda_d)$ and $Y_n = p(X_n)$, $n \in  \mathbb{N}$,
	this reference states that $\mathbb{P}^{X_n \mid Y_n} = 1/\ell(Y_n) \lambda_d\vert_{\mathsf{L}(Y_n)}$.
	The author believes that instead this is supposed to mean that $\mathbb{P}^{X_n \mid Y_n}$ has a density with respect to the Hausdorff measure on $p^{-1}(Y_n)$, 
	as implicitly implied by \cite{roberts2002polar}.
	Nonetheless, as \cite{roberts2001markov} observe, a slice sampler is indeed de-initialized by its level chain, see also proof of Lemma \ref{L: Wasserstein de-initilaization for slice sampling}.} %
that $\mathtt{L}$ can be used for de-initialization of slice sampling,
and its convergence behaviour in total variation distance can be described in terms of $\mathtt{L}$.
Moreover, \cite{natarovskii2021quantitative,rudolf2024dimension} show that the spectral gap of a slice sampler is determined by an $\mathbb{R}$-valued transition kernel related to the auxiliary ``levels'' generated in the outer integral of $\mathtt{S}$.
Observe however that this kernel differs from $\mathtt{L}$.

In the following we demonstrate that $\mathtt{L}$ can also be used for Wasserstein de-initialization of slice sampling.
This means we require a pseudo metric on $p(\mathsf{X})$.
The trick is now to choose this pseudo metric in such a way that is useful to obtain Proposition \ref{P: Wasserstein ergodicity of slice sampling}.
It turns out that the family of functions $m_{\alpha, \kappa}$ defined in \eqref{Eq: definition m_alpha_kappa} interpreted as continuous\footnote{Continuity of $m_{\alpha, \kappa}$ is ensured by continuity of $\ell$.} pseudo metrics on $p(\mathsf{X})$ is suited for this purpose.

\begin{lemma}\label{L: Wasserstein de-initilaization for slice sampling}
	Impose Assumption \ref{A: Slice sampling}, and assume that the level set function $\ell: p(\mathsf{X})\to [0, \infty)$ is continuous.
	\begin{enumerate}[label=\alph*)]
		\item  If there exists some $\alpha \in [0, \infty)$, $\kappa \in (0, \infty)$ and constants $C_1, C_2 \in [0, \infty)$ such that
		\begin{enumerate}[label=(\roman*)]
			\item $\int_{(0,1)} m_\alpha(sr, st)\ \lambda_1(\d s)\leq  C_1\cdot m_\alpha(r, t), \qquad \forall r,t \in p(\mathsf{X})$, 
			\item $W_{\mathrm{d}}\left(\frac{1}{\ell(r)} \varrho\vert_{\mathsf{L}(r)}, \frac{1}{\ell(t)} \varrho\vert_{\mathsf{L}(t)}\right)\leq C_2\cdot m_{\alpha, \kappa}(r,t),\qquad \forall r,t \in \mathsf{W}$,
		\end{enumerate}
		then for all $n \in \mathbb{N}$, $\mu \in \mathcal{M}_{\mathsf{X}}$ holds
		$W_{\mathrm{d}} \big(\mu\mathtt{S}^n, \pi\big) \leq C_1 C_2 \cdot W_{m_{\alpha, \kappa}}\big(p(\mu)\mathtt{L}^{n-1}, p(\pi) \big)$.
		\item If there exist $\alpha \in [0, \infty)$, $\kappa \in (0, \infty)$ and $C_3 \in (0, \infty)$ 
		such that the function $g_{\alpha, \kappa} \circ \ell \circ p : \mathsf{X} \to [0, \infty)$ is well-defined and $C_3$-Lipschitz continuous 	
		with respect to $\mathrm{d}$, 
		then 
		$W_{m_{\alpha, \kappa}}(p(\mu)\mathtt{L}^n, p(\pi))\leq C_3 W_{\mathrm{d}}(\mu\mathtt{S}^n, \pi)$ for all
		$\mu\in \mathcal{M}_{\mathsf{X}}, n \in \mathbb{N}$.
	\end{enumerate}	 
\end{lemma}
\begin{proof}
	We apply Lemma \ref{L: extended lumping} for $\mathsf{X} =\mathsf{X}$, $\mathsf{Y} = p(\mathsf{X})$, $K=\mathtt{S}$, $f= p$ and $Q=\mathtt{L}$.
	
	\textit{Ad a):}
	From Lemma \ref{L: extended lumping}.a) we know that the mapping given by $\Pi_{\mathtt{L}}(\mu, n) := p(\mu)\mathtt{L}^{n-1}$ for $\mu \in \mathcal{M}_{\mathsf{X}}$, $n \in \mathbb{N}$ is de-initializing for $\mathtt{S}$.
	We may write the kernel $R$ as
	$R(t, \mathsf{A}) = \int_{(0,1)}\ell(st)^{-1}\varrho\vert_{\mathsf{L}(st)}(\mathsf{A})\ \lambda_1(\d s)$
	for $t \in p(\mathsf{X}), \mathsf{A} \in \mathcal{B}(\mathsf{X})$.
	Hence by \cite[Theorem 4.8]{villani2009optimal}
	\begin{align*}
		&W_{\mathrm{d}}\big( R(r, \cdot), R(t, \cdot)\big)
		\leq \int_{(0,1)} W_{\mathrm{d}}\left(\frac{1}{\ell(sr)} \varrho\vert_{\mathsf{L}(sr)}, \frac{1}{\ell(st)} \varrho\vert_{\mathsf{L}(st)}\right)\ \lambda_1(\d s)\\
		&\qquad\overset{(ii)}{\leq}\int_{(0,1)} C_2\cdot m_{\alpha, \kappa}(sr,st)\ \lambda_1(\d s) 
		\overset{(i)}{\leq} C_2 C_1\cdot m_{\alpha, \kappa}(r,t), \qquad \forall r,t \in p(\mathsf{X}),
	\end{align*}
	such that $\Pi_{\mathtt{L}}$ is also Wasserstein de-initializing for $\mathtt{S}$ with respect to the pseudo metrics $\mathrm{d}$ and $m_{\alpha, \kappa}$.
	
	\textit{Ad b):} From Lemma \ref{L: extended lumping}.b) we know that the mapping given by $\Pi_{\mathtt{S}}(\mu,n):= \mu R\mathtt{S}^{n-1}$ for $\mu \in \mathcal{M}_{p(\mathsf{X})}$, $n \in \mathbb{N}$ is de-initializing for $\mathtt{L}$.
	Our assumption yields $	m_{\alpha, \kappa}\big(p(x), p(y)\big)\leq C_3\cdot \mathrm{d}(x,y)$ for all $ x,y \in \mathsf{X}$,
	such that $\Pi_{\mathtt{S}}$ is also Wasserstein de-initializing for $\mathtt{L}$.
	
	Overall, the Wasserstein inequalities follow from Theorem \ref{Thm: Wasserstein de-initialization} and the invariance of $\mathtt{S}$ and $\mathtt{L}$ with respect to $\pi$ and $p(\pi)$, respectively, and $p(\mu)R = \mu\mathtt{S}$ for $\mu \in \mathcal{M}_{\mathsf{X}}$ in part b).
\end{proof}

To prove Proposition \ref{P: Wasserstein ergodicity of slice sampling}, we combine the previous statement about Wasserstein de-initialization with an explicit calculation of the Wasserstein distance of the lumped kernel $\mathtt{L}$.

\begin{lemma}\label{L: explict formula for m_alpha-Wasserstein distance}
	Assume that the level set function $\ell: p(\mathsf{X}) \to [0, \infty)$ is continuous and strictly decreasing with $\lim_{t \to \|p\|_{\infty}} \ell(t)= 0$.
	Moreover, let $g: \ell(\mathsf{W}) \to \mathbb{R}$ be a strictly increasing, continuously differentiable function,
	with non-zero derivative on the interior of $\ell(\mathsf{W})$, and
	$\lim_{t \mapsto 0} g(t) \cdot t = 0$.
	Assume that the antiderivative $G$ of $g$ also satisfies
	$\lim_{z \to 0} G(z) = 0$.
	Set $m: p(\mathsf{X})^2 \to [0, \infty), (t,r) \mapsto |(g\circ \ell)(t) - (g\circ \ell)(r)|$.
	We have
	\[
	W_m\big(\mathtt{L}(t, \cdot), \mathtt{L}(r,\cdot)\big) 
	= \left|\int_{(0,1)} \frac{G\big(\ell(sr)\big)}{\ell(sr)}  - \frac{G\big(\ell(st)\big)}{\ell(st)} \ \lambda_1(\d s)\right|, \qquad \forall\ t,r \in p(\mathsf{X}).
	\]
\end{lemma}
\begin{proof}
	By virtue of \cite{rueschendorf1985Wasserstein} and \cite{vallender1973calculation} (see Remark \ref{R: How to apply Rueschendorf} with $\eta = \mathtt{L}(t, \cdot)$ and $\xi=\mathtt{L}(r, \cdot)$) we know that
	\begin{equation}\label{Eq: explicit expression 1-dim Wasserstein distance}
			W_m\big(\mathtt{L}(t, \cdot), \mathtt{L}(r,\cdot)\big)  
			= \int_{\mathbb{R}} \left| H_t(x) - H_r(x)\right|\ \lambda_1(\d x),\qquad \forall t,r \in p(\mathsf{X}),
	\end{equation}
	where $H_s$ denotes the cumulative distribution function of the pushforward measure  $(g\circ \ell)\big(\mathtt{L}(s, \cdot)\big)$ of $\mathtt{L}(s, \cdot)$ under $g \circ \ell$.
	We first determine $H_t$ for $t \in p(\mathsf{X})$.
	To this end set 
	$c_{\inf}:= \inf_{s \in \ell(p(\mathsf{X}))} g(s)$ and $c_{\sup}:= \sup_{s \in \ell(p(\mathsf{X}))} g(s)$.
	Note that $g$ and $g\circ\ell$ are invertible on $\ell(p(\mathsf{X}))$ and $p(\mathsf{X})$, respectively, due to their strict monotonicity.
		For $x \in (c_{\inf}, c_{\sup})$ we have
	\begin{align*}
		H_t(x) &= \mathtt{L}\big(t,\left[(g \circ \ell)^{-1}(x), \infty\right) \cap p(\mathsf{X})\big)
		= \frac{1}{t} \int_{(0,t)} \frac{1}{\ell(s)} \varrho\Big(\mathsf{L}(s) \cap \mathsf{L}\big((g \circ \ell)^{-1}(x)\big)\Big) \ \lambda_1(\d s)\\
		&= \frac{1}{t} \int_{(0,t)} \frac{1}{\ell(s)} \Big(\ell(s) \wedge \ell\left((g \circ \ell)^{-1}(x)\right)\Big)\ \lambda_1(\d s)
		= \frac{1}{t} \int_{(0,t)} \frac{1}{\ell(s)}\left( \ell(s) \wedge g^{-1}(x)\right)\ \lambda_1(\d s)\\
		&= \frac{1}{t} \int_{(0,t)} 1 \wedge \frac{g^{-1}(x)}{\ell(s)}\ \lambda_1(\d s)
		=   \int_{(0,1)} 1 \wedge \frac{g^{-1}(x)}{\ell(st)}\ \lambda_1(\d s),
	\end{align*}
	where we used that $\{z \in \mathsf{X} \mid p(z)\geq (g \circ \ell)^{-1}(x)\}$ and $\mathsf{L}((g \circ \ell)^{-1}(x))$ differ only by a $\varrho$-null set due to the continuity of $\ell$. 
	Moreover, we have $H_t(x)= \mathtt{L}(t ,\emptyset)= 0$ for $x  < c_{\inf}$, and $H_t(x)= \mathtt{L}(t ,p(\mathsf{X}))= 1$  for $x \geq c_{\sup}$. 
	 
	Fix now $t,r \in \mathsf{W}$ with $t > r$.
	Since $\ell$ is decreasing, this implies $1/\ell(st) \geq 1/\ell(sr)$ for all $s \in (0,1)$.
	Plugging the formula for $H_t$ and $H_r$ into \eqref{Eq: explicit expression 1-dim Wasserstein distance} we get
	\begin{align*}
		&W_m\big(\mathtt{L}(t, \cdot), \mathtt{L}(r,\cdot)\big) 
		= \int_{(c_{\inf}, c_{\sup})} \left| \int_{(0,1)} \left(1 \wedge \frac{g^{-1}(x)}{\ell(st)}\right) - \left(1 \wedge \frac{g^{-1}(x)}{\ell(sr)}\right)\ \lambda_1(\d s)\right|\ \lambda_1(\d x)\\
		&\qquad=
			\begin{multlined}[0.7\textwidth][t]
				\int_{(0,1)} \int_{(c_{\inf}, c_{\sup})} \mathbbm{1}_{\{\frac{g^{-1}(x)}{\ell(sr)} < 1 \leq \frac{g^{-1}(x)}{\ell(st)}\}} \left(1 - \frac{g^{-1}(x)}{\ell(sr)}\right)\\
			+ \mathbbm{1}_{\{\frac{g^{-1}(x)}{\ell(st)} < 1\}}\left(\frac{g^{-1}(x)}{\ell(st)} - \frac{g^{-1}(x)}{\ell(sr)}\right)\ \lambda_1(\d x)\,\lambda_1(\d s).
			\end{multlined}
	\end{align*}
	Computing the inner integrals one by one yields
	\[
		\int_{(c_{\inf}, c_{\sup})} \mathbbm{1}_{\{\frac{g^{-1}(x)}{\ell(sr)} < 1 \leq \frac{g^{-1}(x)}{\ell(st)}\}}\ \lambda_1(\d x)
		=\lambda_1\Big(\big[g(\ell(st)), g(\ell(sr)\big)\Big)
		= g\big(\ell(sr)\big) - g\big(\ell(st)\big),
	\]
	and, using a conversion to the Riemann integral, the substitution $y = g^{-1}(x)$ together with the inverse function theorem for the derivative of $g^{-1}$ and integration by part,
	\begin{align*}
		&\int_{(c_{\inf}, c_{\sup})} \mathbbm{1}_{\{\frac{g^{-1}(x)}{\ell(sr)} < 1 \leq \frac{g^{-1}(x)}{\ell(st)}\}} \frac{g^{-1}(x)}{\ell(sr)}\ \lambda_1(\d x)
		=\frac{1}{\ell(sr)}\int_{g(\ell(st))}^{g(\ell(sr))}g^{-1}(x)\ \d x\\
		&\qquad= \frac{1}{\ell(sr)} \int_{\ell(st)}^{\ell(sr)} y \cdot g'(y) \d y
		= \frac{1}{\ell(sr)} \left( \big[ y \cdot g(y) \big]_{\ell(st)}^{\ell(sr)} - \int_{\ell(st)}^{\ell(sr)} g(y) \d y \right)\\
		&\qquad = g\big(\ell(sr)\big) - \frac{\ell(st)}{\ell(sr)} g\big(\ell(st)\big) - \frac{1}{\ell(sr)} \int_{\ell(st)}^{\ell(sr)} g(y) \d y,
	\end{align*}
	and, with similar arguments using additionally $c_{\inf}=\lim_{z \to 0} g(z)$ and $\lim_{t \mapsto 0} g(t) \cdot t = 0$,
	\begin{align*}
		&\int_{(c_{\inf}, c_{\sup})}\mathbbm{1}_{\{\frac{g^{-1}(x)}{\ell(st)} < 1\}}\frac{g^{-1}(x)}{\ell(st)}\ \lambda_1(\d x)
		=\frac{1}{\ell(st)}\lim_{z \to 0}\int_{g(z)}^{g(\ell(st))}g^{-1}(x)\ \d x\\
		&\qquad= \frac{1}{\ell(st)} \left( \ell(st) \cdot g(\ell(st)) - \lim_{y \to 0} y\cdot g(y) - \int_{0}^{\ell(st)} g(y) \d y \right)\\
		&\qquad= g(\ell(st)) - \frac{1}{\ell(st)} \int_{0}^{\ell(st)} g(y) \d y,
	\end{align*}
	and 
	\[
		\int_{(c_{\inf}, c_{\sup})}\mathbbm{1}_{\{\frac{g^{-1}(x)}{\ell(sr)} < 1\}}\frac{g^{-1}(x)}{\ell(st)}\ \lambda_1(\d x)
		= \frac{\ell(st)}{\ell(sr)}g(\ell(st)) - \frac{1}{\ell(sr)} \int_{0}^{\ell(st)} g(y) \d y.
	\]
	Overall we obtain
	\begin{align*}
		&W_m\big(\mathtt{L}(t, \cdot), \mathtt{L}(r,\cdot)\big)
	 = \int_{(0,1)} \frac{1}{\ell(sr)} \int_{0}^{\ell(sr)} g(y) \d y   - \frac{1}{\ell(st)} \int_{0}^{\ell(st)} g(y) \d y\ \lambda_1(\d s)\\
	 	&\qquad = \int_{(0,1)} \frac{G\big(\ell(sr)\big)}{\ell(sr)}    - \frac{G\big(\ell(st)\big)}{\ell(st)}\ \lambda_1(\d s),
	\end{align*}
	where we use for the second equality $\lim_{z \to 0} G(z) = 0$.
	The statement of the theorem follows, because $\ell(\mathsf{W}) \to \mathbb{R}, z \mapsto G(z)/z$ is monotonically increasing,
	as can be seen from its derivative $z \mapsto (g(z) z - G(z))/z^2$ and $G(z) \leq g(z)z$ for $z \in \ell(\mathsf{W})$ by virtue of the fundamental theorem of calculus and the monotonicity of $g$.
\end{proof}

With the previous lemma, we can show that the level kernel $\mathtt{L}$ is Wasserstein contracting under the assumptions of Proposition \ref{P: Wasserstein ergodicity of slice sampling}.
In combination with the Wasserstein bounds from Lemma \ref{L: Wasserstein de-initilaization for slice sampling}, this yields Proposition \ref{P: Wasserstein ergodicity of slice sampling}.

\begin{proof}[Proof of Proposition \ref{P: Wasserstein ergodicity of slice sampling}]
		Observe that for all choices of $\alpha$ and $\kappa$ the assumptions of Lemma \ref{L: explict formula for m_alpha-Wasserstein distance} for $g= g_{\alpha, \kappa}$ are satisfied.
		Denote by $G_{\alpha, \kappa}$ the antiderivative of $g_{\alpha, \kappa}$.
		We have
		\[
			\frac{G_{\alpha, \kappa}(x)}{x} = 
			\begin{dcases}
				\frac{1}{\alpha +1} g_{\alpha, \kappa}(x),& \alpha \in (0, \infty),\\
				\frac{1}{\alpha +1} g_{\alpha, \kappa}(x) - 1, & \alpha = 0,
			\end{dcases}
			\qquad x\in (0, \infty).
		\]
		Hence, we get by Lemma \ref{L: explict formula for m_alpha-Wasserstein distance}
		\begin{align*}
			&W_{m_{\alpha, \kappa}}\big(\mathtt{L}(t, \cdot), \mathtt{L}(r, \cdot)\big)
			=  \frac{1}{\alpha +1}\left|\int_{(0,1)} g_{\alpha, \kappa}\big(\ell(sr)\big) - g_{\alpha, \kappa}\big(\ell(st)\big) \ \lambda_1(\d s)\right|\\
			&\qquad=  \frac{1}{\alpha +1}\int_{(0,1)} \left|g_{\alpha, \kappa}\big(\ell(sr)\big) - g_{\alpha, \kappa}\big(\ell(st)\big)\right| \ \lambda_1(\d s)
			=  \frac{1}{\alpha +1}\int_{(0,1)} m_{\alpha, \kappa}(sr,st) \ \lambda_1(\d s)\\
			&\qquad\overset{(i)}{\leq} \frac{\beta}{\alpha +1} m_{\alpha, \kappa}(r,t),
			\qquad \qquad \forall r,t \in p(\mathsf{X}),
		\end{align*}
		where for the second equality we exploit that $g_{\alpha, \kappa} \circ \ell$ is strictly decreasing, such that the integrand is either a non-negative or non-positive function.
		This implies $W_{m_{\alpha, \kappa}}\big(\nu\mathtt{L}^n, p(\pi)\big)\leq (\beta/(1+\alpha))^n W_{m_{\alpha,\kappa}}(\nu,p(\pi))$ for all $n \in \mathbb{N}$, $\nu \in \mathcal{M}_{p(\mathsf{X})}$, see e.g.\ \cite[Lemma 20.3.2 and Theorem 20.3.4]{douc2018markov}.
		Therefore by Lemma \ref{L: Wasserstein de-initilaization for slice sampling}
		\begin{align*}
			W_{\mathrm{d}}(\mu\mathtt{S}^n, \pi) 
			&\leq C \beta \cdot W_{m_{\alpha, \kappa}}\big(p(\mu)\mathtt{L}^{n-1}, p(\pi)\big)
			\leq C \beta \left(\frac{\beta}{1+\alpha}\right)^{n-1}\cdot W_{m_{\alpha,\kappa}}\big(p(\mu),p(\pi)\big)\\
			&= C (1+\alpha) \left(\frac{\beta}{1+\alpha}\right)^{n}\cdot W_{m_{\alpha,\kappa}}\big(p(\mu),p(\pi)\big),
			\qquad \forall \mu \in \mathcal{M}_{\mathsf{X}}, n \in \mathbb{N}.
		\end{align*}
\end{proof}

We add several examples how Proposition \ref{P: Wasserstein ergodicity of slice sampling} can be applied.
Throughout the rest of this section, we denote by $B_r(x) := \{y \in \mathsf{X} \mid \mathrm{d}(x,y) < r\}$ the open metric ball with radius $r \in[0, \infty)$ and centre point $x$ in the respective metric spaces.
Moreover, let  $ \omega_{d}$ be the volume of the $d$-dimensional Euclidean unit sphere $\mathbb{S}^d \subseteq\mathbb{R}^{d+1}$.

\begin{example}[Uniform slice sampling on $\mathbb{R}^d$]\label{Ex: Uniform slice sampling on Rd}
	Fix $d \in \mathbb{N}$. We consider uniform slice sampling on $(\mathbb{R}^d, \mathcal{B}(\mathbb{R}^d))$, i.e.\ we use reference measure $\varrho = \lambda_d$,
	and equip $\mathbb{R}^d$ with the Euclidean metric $\|\cdot\|$.
	Let $\Sigma \in \mathbb{R}^{d\times d}$ be a symmetric positive definite matrix and let $x_0 \in \mathbb{R}^d$.
	We consider unnormalized densities of the form
	\[
		p(x) = \exp\big(- \varphi(\|\Sigma^{1/2} x - x_0\|)\big), \qquad x \in \mathbb{R}^d,
	\]
	where $\varphi: [0, \infty) \to \mathbb{R}$ is a continuous, strictly increasing, convex function such that $p$ is $\lambda_d$-integrable.
	This requires in particular that $\lim_{x \to \infty}\varphi(x) = \infty$, and we have $p(\mathbb{R}^d) = \mathsf{W} = (0,e^{-\varphi(0)}]$.
	Then the level sets are of the form
	$\mathsf{L}(t) = \left\{x \in \mathbb{R}^d \mid \|\Sigma^{1/2}x - x_0\| < \varphi^{-1}\big(-\ln(t)\big)\right\}$, $t  \in \mathsf{W}$,
	and we have
	\[
		\ell(t) 
		= \frac{1}{\sqrt{\det(\Sigma)}}\int_{\mathbb{R}^d} \mathbbm{1}_{(0,\varphi^{-1}(-\ln(t)) )}\big(\|x\|\big)\ \lambda_d(\d x)
		= \varphi^{-1}\big(-\ln(t)\big)^d \cdot \frac{\lambda_d(B_1(0))}{\sqrt{\det(\Sigma)}}, \qquad t \in \mathsf{W}.
	\]
	
	Using for $r,t \in \mathsf{W}$ the coupling
	\[
	\begin{multlined}
		\eta(\mathsf{C}) := \frac{1}{\lambda_d(B_1(0))}\int_{B_1(0)}  
		\mathbbm{1}_{\mathsf{C}}\Big(\Sigma^{-1/2}\big(\varphi^{-1}(-\ln(t))x+x_0\big),\Sigma^{-1/2}\big( \varphi^{-1}(-\ln(r))x+x_0\big)\Big)\\
		\times \lambda_d(\d x),
		\qquad \mathsf{C} \in \mathcal{B}(\mathbb{R}^{2d}),
	\end{multlined}
	\]
	of $1/\ell(t)\lambda_d\vert_{\mathsf{L}(t)}$ and $1/\ell(r)\lambda_d\vert_{\mathsf{L}(r)}$,
	we get
	\begin{align*}
		&W_{\|\cdot\|}\left(\frac{1}{\ell(t)}\lambda_d\vert_{\mathsf{L}(t)}, \frac{1}{\ell(r)}\lambda_d\vert_{\mathsf{L}(r)}\right)
		\leq \int_{\mathbb{R}^{2d}} \|x-y\|\ \eta(\d(x,y))\\
		&\qquad\leq |\varphi^{-1}(-\ln(t)) - \varphi^{-1}(-\ln(r))| \frac{1}{\lambda_d(B_1(0))}\int_{B_1(0)}  
							\left\|\Sigma^{-1/2}x\right\|\ \lambda_d(\d x)\\
		&\qquad\leq \frac{\|\Sigma^{-1/2}\|}{\lambda_d(B_1(0))}\int_{B_1(0)}  \left\|x\right\|\ \lambda_d(\d x)\cdot|\varphi^{-1}(-\ln(t)) - \varphi^{-1}(-\ln(r))|\\
		&\qquad \leq \|\Sigma^{-1/2}\| \frac{d}{d+1}\cdot \left|\left(\frac{\sqrt{\det(\Sigma)}}{\lambda_d(B_1(0))} \ell(t)\right)^{1/d} - \left(\frac{\sqrt{\det(\Sigma)}}{\lambda_d(B_1(0))} \ell(r)\right)^{1/d}\right|,
	\end{align*}
	i.e.\ (ii) in Proposition \ref{P: Wasserstein ergodicity of slice sampling} is satisfied for $\alpha = 1/d$, $C = \|\Sigma^{-1/2}\|d/(d+1)$ and 
	$\kappa = \lambda_d(B_1(0))/ \sqrt{\det(\Sigma)}$.
	
	Furthermore, applying Lemma \ref{L: concavity considerations} with $\phi = \varphi^{-1}$ and $h = -\ln$, we get
	\[
		\left|\varphi^{-1}\big(-\ln(sr)\big) - \varphi^{-1}\big(-\ln(st)\big) \right|\leq \left|\varphi^{-1}\big(-\ln(r)\big) - \varphi^{-1}\big(-\ln(t)\big)\right|
	\]
	 for $s \in (0,1), r,t \in \mathsf{W}$.
	Hence
	$\int_{(0,1)} |\ell(sr)^{1/d} - \ell(st)^{1/d}|\ \lambda_1(\d s) \leq |\ell(r)^{1/d} - \ell(t)^{1/d}|$
	and (i) in Proposition \ref{P: Wasserstein ergodicity of slice sampling} holds for $\alpha = 1/d$ and $\beta = 1$.
	
	Since $g_{\alpha, \kappa}\circ \ell \circ p(x) = \|\Sigma^{1/2}x -x_0\|$, $x \in \mathbb{R}^d$,
	is $\|\Sigma^{1/2}\|$-Lipschitz continuous, we get uniform Wasserstein ergodicity by \eqref{Eq: Slice sampler Wasserstein ergodicity} with constants
	\[
		W_{\|\cdot\|}\big(\mu\mathtt{S}^n, \pi\big)\leq \|\Sigma^{1/2}\| \|\Sigma^{-1/2}\|\cdot \left(\frac{d}{d+1}\right)^n\cdot W_{\|\cdot\|}(\mu, \pi), \qquad n \in \mathbb{N}, \mu \in \mathcal{M}_{\mathbb{R}^d}.
	\]
	If $\Sigma$ is the identity matrix, i.e.\ if the unnormalized target density $p$ is isotropic, we recover \cite[Theorem 2.1]{natarovskii2021quantitative}.
\end{example}

\begin{example}[Uniform slice sampling on hyperbolic space]
	Fix $d \geq 2$.
	We consider slice sampling on $d$-dimensional hyperbolic space $(\mathcal{H}_d,\mathcal{B}(\mathcal{H}_d))$, i.e.\ a complete simply connected Riemannian manifold of constant curvature -1, equipped with its Riemannian measure $\nu$ and the metric $\mathrm{d}$ induced by its Riemannian metric, see \cite[Sections II.1, II.5 and IV.1]{sakai1996riemannian}.
	Fix a point $o \in \mathcal{H}_d$ as origin.
	We assume that the level sets of the unnormalized target density $p:\mathcal{H}_d \to [0, \infty)$ are metric balls around $o$.
	To enforce $p > 0$, we use $\mathsf{X} = p^{-1}((0, \infty))$ in the following.
	Using a decreasing function $f: \mathsf{W} \to [0, \infty)$ to describe the radii of the level sets, we obtain from Lemma \ref{L: General level set computations} that $\ell(t) = \omega_{d-1}\Psi(f(t))$, $t \in \mathsf{W}$,
	where $\Psi: [0, \infty)\to [0, \infty), x \mapsto \int_{(0, x)} \sinh^{d-1}(u)\ \lambda_{1}(\d u)$,
	and
	\begin{align*}
		&W_{\mathrm{d}}\left(\frac{1}{\ell(r)}\nu\vert_{\mathsf{L}(r)}, \frac{1}{\ell(t)}\nu\vert_{\mathsf{L}(t)}\right)
		= \left| \frac{\int_{(0,f(r))} u \sinh^{d-1}(u)\ \lambda_1(\d u)}{\Psi(f(r))}
		- \frac{\int_{(0,f(t))} u \sinh^{d-1}(u)\ \lambda_1(\d u)}{\Psi(f(t))} \right|\\
		&\qquad \overset{\text{Lemma } \ref{L: bound for balls in hyperbolic space}}{\leq}
		\left|\ln\big(\Psi(f(r))\big)- \ln\big(\Psi(f(t))\big)\right|
		= \left|\ln(\ell(r))- \ln(\ell(t))\right|, \qquad \forall r,t \in \mathsf{W}.
	\end{align*}
	Hence, (ii) in Proposition \ref{P: Wasserstein ergodicity of slice sampling} holds with $\alpha = 0$ and $C = 1$.
	
	If the unnormalized density takes the form $p(x) =\varphi(\mathrm{d}(x,o)) \cdot \mathbbm{1}_{B_c(o)}(x)$ for some $c \in [0, \infty)$ and a strictly decreasing, continuous, convex function $\varphi: [0, c) \to [0, \infty)$ with $\lim_{x \to 0} \varphi(x)= \infty$ and $\lim_{x \to c} \varphi(x) = 0$,
	then $f= \varphi^{-1}$ and $p(\mathsf{X}) = \mathsf{W} = (0, \infty)$.
	Since $\Psi$ is log-concave (see \cite[Theorem 2(i)]{zhang2012some}) and increasing, and $\varphi^{-1}$ is concave, also $\ln \circ \Psi \circ \varphi^{-1}$ is concave.
	From Lemma \ref{L: concavity considerations} (with $\phi = \ln \circ \Psi \circ \varphi^{-1}$, $h(x)= x$) we get
	\[
		\int_{(0,1)} \left|\ln\big(\ell(st)\big) - \ln\big(\ell(st)\big)\right|\ \lambda_1(\d s)
		\leq
		\frac{1}{2} \cdot \left|\ln\big(\ell(t)\big) - \ln\big(\ell(t)\big)\right|, \qquad \forall r,t \in (0, \infty),
	\] 
	such that (i) in Proposition \ref{P: Wasserstein ergodicity of slice sampling} is satisfied with $\alpha = 0$ and $\beta =1/2$.
	We get $W_{\mathrm{d}}(\mu \mathtt{S}^n, \pi) \leq (1/2)^n W_{m_0}(p(\mu), p(\pi))$ for all $\mu \in \mathcal{M}_{\mathsf{X}}$, $n \in \mathbb{N}$.
	Note that in this case Proposition \ref{P: Wasserstein ergodicity of slice sampling} likely gives us a non-uniform bound, because $g_{0,1}\circ \ell \circ p $ is not Lipschitz.
\end{example}

\begin{example}[Uniform slice sampling on $\mathbb{S}^{d}$]\label{Ex: Uniform slice sampling on the sphere}
	For $d \geq 2$, we consider slice sampling on $(\mathbb{S}^{d}, \mathcal{B}(\mathbb{S}^{d}), \sigma)$, where $\sigma$ is the volume measure on $\mathbb{S}^{d}$.
	We use the great circle distance $\mathrm{d}(x,y) := \arccos(x^Ty)$, $x,y \in \mathbb{S}^d$, where $x^T$ denotes the transpose of the vector $x$.
	Fix a point $o \in \mathbb{S}^d$, that serves as origin, and an unnormalized target density $p: \mathbb{S}^d \to (0, \infty)$.
	In the following, we assume that the level sets of the target distribution $\pi(\d x) \propto p(x) \sigma(\d x)$ are metric balls around $o$,
	whose radii are described by a function $f:\mathsf{W} \to [0,\pi)$.
	Set $\Psi(t):= \int_{(0,t)} \sin^{d-1}(u)\ \lambda_1(\d u)$ for $t \in [0, \pi)$.
	Applying Lemma \ref{L: General level set computations},
	we get $\ell(t) = \omega_{d-1}\Psi(f(t))$, $t \in [0, \pi)$, and
	\begin{align*}
		&W_{\mathrm{d}}\left(\frac{1}{\ell(t)}\sigma\vert_{\mathsf{L}(t)}, \frac{1}{\ell(r)}\sigma\vert_{\mathsf{L}(r)}\right)
		\leq \left|\frac{\int_0^{f(t)} u \sin^{d-1}(u)\ \d u}{\Psi(f(t))} - \frac{\int_0^{f(r)} u \sin^{d-1}(u)\ \d u}{\Psi(f(r))}\right|\\
		&\qquad\overset{\text{Lemma } \ref{L: bounds for balls on the sphere}}{\leq} \frac{d\pi\sqrt{2}}{2^{3/(2d)}} \cdot\left|\Psi\big(f(t)\big)^{1/d} - \Psi\big(f(r)\big)^{1/d}\right|
		= \frac{d\pi\sqrt{2}}{2^{3/(2d)}} \cdot\left|\left(\frac{\ell(t)}{\omega_{d-1}}\right)^{1/d}  - \left(\frac{\ell(r)}{\omega_{d-1}}\right)^{1/d} \right|
	\end{align*}
	for all $r,t \in \mathsf{W}$.
	Hence, (ii) in Proposition \ref{P: Wasserstein ergodicity of slice sampling} holds with $\alpha = 1/d$ and $C  = d\pi2^{1/2 - 3/(2d)}$.
	
	Assume the unnormalized target density is of the from $p(x)= \varphi(\mathrm{d}(x,o))$, $x \in \mathbb{S}^d$, for some continuous, strictly decreasing convex function $\varphi: [0,\pi] \to [0, \infty)$.
	Set $b := \varphi(0)$ and $a := \varphi(\pi)$.
	Then 
	\[
		\left(\frac{\ell(t)}{\omega_{d-1}}\right)^{1/d} = \begin{dcases}
															\Psi\big(\varphi^{-1}(t)\big)^{1/d},& t \in [a,b],\\
															\Psi(\pi)^{1/d}, &, t \in (0, a).
		\end{dcases}
	\]
	This function is concave on $[a,b]$, because $x \mapsto \Psi(x)^{1/d}$ is concave and increasing (see Lemma \ref{L: bounds for balls on the sphere}) and $\varphi^{-1}$ is also concave.
	Extending this decreasing function to the left by its maximum preserves the concavity.
	We get from Lemma \ref{L: concavity considerations} (with $\phi(x) =(\ell(x)/\omega_{d-1})^{1/d}$ and $h(x)= x$)
	\begin{align*}
		&\int_{(0,1)} \left|\left(\frac{\ell(st)}{\omega_{d-1}}\right)^{1/d} - \left(\frac{\ell(sr)}{\omega_{d-1}}\right)^{1/d} \right| \ \lambda_1(\d s)
		\leq  \left|\left(\frac{\ell(t)}{\omega_{d-1}}\right)^{1/d} - \left(\frac{\ell(r)}{\omega_{d-1}}\right)^{1/d} \right| \int_{(a/(r \wedge t), 1)} s\ \lambda_1(\d s) \\
		&\qquad\leq \frac{1}{2}\left(1-\frac{a^2}{b^2}\right) \cdot \left|\left(\frac{\ell(t)}{\omega_{d-1}}\right)^{1/d} - \left(\frac{\ell(r)}{\omega_{d-1}}\right)^{1/d} \right|, \qquad \forall r,t\in [a,b],
	\end{align*}
	such that Proposition \ref{P: Wasserstein ergodicity of slice sampling}.(i) is satisfied for $\alpha = 1/d$ and $\beta =(1-a^2/b^2)/2$.
	By the Lipschitzness statement in Lemma \ref{L: bounds for balls on the sphere},
	we get from \eqref{Eq: Slice sampler Wasserstein ergodicity} 
	\[
		W_{\mathrm{d}}\big(\mu \mathtt{S}^n, \pi\big) \leq \frac{\pi \sqrt{2} (d+1)}{(d 2^{3/2})^{1/d}} \cdot \left(\frac{1}{2} \cdot\left(1 - \frac{a^2}{b^2}\right)\cdot \frac{d}{d+1}\right)^n
		W_{\mathrm{d}}\big(\mu , \pi), \qquad \forall n \in \mathbb{N}, \mu \in \mathcal{M}_{\mathbb{S}^d}.
	\]
 	This means for this class of distributions we get a dimension independent Wasserstein convergence rate of at least $1/2$.
 	Additionally, the closer the ratio $a/b$ is to one, the smaller is the factor $1-a^2/b^2$, giving an even smaller (dimension independent) Wasserstein convergence rate.
 	Intuitively, this factor could be interpreted as capturing how close $\pi$ is to the uniform distribution on $\mathbb{S}^d$ in terms of density range.
 	Of course in the leading constant we do observe a dependence on the dimension, which is approximately linear.
\end{example}

The application of Proposition \ref{P: Wasserstein ergodicity of slice sampling} is not limited to unimodal distributions, as the following example demonstrates.

\begin{example}[Bimodal distribution on the sphere]
	Consider again uniform slice sampling on $(\mathbb{S}^{d}, \mathcal{B}(\mathbb{S}^{d}), \sigma)$ as described in Example \ref{Ex: Uniform slice sampling on the sphere}.
	Let the unnormalized target density be of the form $p(x) = \exp(\varphi(|x^To|))$, where $\varphi: [0,1] \to \mathbb{R}$ is a bounded, continuous, strictly decreasing, concave function and $o \in \mathbb{S}^d$.
	A distribution $\pi$ with $\pi(\d x) \propto p(x) \sigma(\d x)$ has two modes at the antipodal points $o$ and $-o$.
	Observe that $\mathsf{W} = (0, e^{\varphi(0)}]$.
	Set 
	\[
		f: (0, e^{\varphi(0)}] \to \left[0, \frac{\pi}{2}\right], \qquad t \mapsto \begin{dcases}
																					\arccos\Big(\varphi^{-1}\big(\ln(t)\big)\Big), & t \in  [e^{\varphi(1)}, e^{\varphi(0)}]\\
																			\frac{\pi}{2}, & t \in  (0, e^{\varphi(1)}).
		\end{dcases}
	\]
	Rewriting $p$ in terms of the spherical distance, we get that the level sets are of the form $\mathsf{L}(t) = B_{f(t)}(o) \cup B_{f(t)}(-o)$  for $t \in \mathsf{W}$.
	Hence, Lemma \ref{L: General level set computations} yields $\ell(t) = 2 \omega_{d-1} \Psi(f(t))$ for $t \in \mathsf{W}$, 
	where $\Psi$ is defined as in Example \ref{Ex: Uniform slice sampling on the sphere}.
	Exploiting that we can rewrite $1/\ell(t) \sigma\vert_{\mathsf{L}}(t) = 1/\ell(t) \sigma\vert_{B_{f(t)}(o)} + 1/\ell(t) \sigma\vert_{B_{f(t)}(-o)}$,
	we get that $1/2(\eta_1 + \eta_2)$ is a coupling of $1/\ell(t) \sigma\vert_{\mathsf{L}(t)}$ and $1/\ell(r) \sigma\vert_{\mathsf{L}(r)}$ with $\eta_i \in \mathcal{C}(2/\ell(t) \sigma\vert_{B_{f(t)}((-1)^io)},2/\ell(r) \sigma\vert_{B_{f(t)}((-1)^io)}) $ being optimal for $i \in \{1,2\}$.
	Combining this with Lemma \ref{L: General level set computations} and Lemma \ref{L: bounds for balls on the sphere} as in Example \ref{Ex: Uniform slice sampling on the sphere} this yields
	\[
		W_{\mathrm{d}}\left(\frac{1}{\ell(t)}\sigma\vert_{\mathsf{L}(t)}, \frac{1}{\ell(r)}\sigma\vert_{\mathsf{L}(r)}\right)
		\leq\frac{d\pi\sqrt{2}}{2^{3/(2d)}} \cdot\left|\left(\frac{\ell(t)}{2\omega_{d-1}}\right)^{1/d}  - \left(\frac{\ell(r)}{2\omega_{d-1}}\right)^{1/d} \right|,
		\qquad \forall r,t \in \mathsf{W}.
	\]
	Since $\arccos$ is concave and decreasing on $[0,1]$ and $\varphi^{-1}$ is convex by assumption, $\arccos\circ\varphi^{-1}$ is concave and decreasing, and may be extended to $\ln(\mathsf{W})$ similar as in Example \ref{Ex: Uniform slice sampling on the sphere}.
	Hence the concatenation with the concave increasing function $x \mapsto \Psi(x)^{1/d}$ (see Lemma \ref{L: bounds for balls on the sphere}) is concave and decreasing.
	Therefore, an application of Lemma \ref{L: concavity considerations} (with $\phi(x) = 2\Psi(f(e^{x}))^{1/d}$ and $h = \ln$) yields
	\begin{align*}
		&\int_{(0,1)} \left|\left(\frac{\ell(st)}{2\omega_{d-1}}\right)^{1/d} - \left(\frac{\ell(sr)}{2\omega_{d-1}}\right)^{1/d} \right| \ \lambda_1(\d s)
		\leq  \left|\left(\frac{\ell(t)}{2\omega_{d-1}}\right)^{1/d} - \left(\frac{\ell(r)}{2\omega_{d-1}}\right)^{1/d} \right| \lambda_1\bigg(\Big(\frac{e^{\varphi(1)}}{r \wedge t}, 1\Big)\bigg) \\
		&\qquad\leq \left(1 - e^{\varphi(1) - \varphi(0)}\right) \cdot \left|\left(\frac{\ell(t)}{2\omega_{d-1}}\right)^{1/d} - \left(\frac{\ell(r)}{2\omega_{d-1}}\right)^{1/d} \right|, \qquad \forall r,t\in [e^{\varphi(1)}, e^{\varphi(0)}].
	\end{align*}
	Overall, we obtain that Proposition \ref{P: Wasserstein ergodicity of slice sampling} is satisfied with $\alpha = 1/d$, $\beta = 1 - e^{\varphi(1) - \varphi(0)}$, $\kappa = 2\omega_{d-1}$ and $C= d\pi2^{1/2-3/(2d)}$.
	As in Example \ref{Ex: Uniform slice sampling on the sphere}, we get from \eqref{Eq: Slice sampler Wasserstein ergodicity} 
	\[
		W_{\mathrm{d}}\big(\mu \mathtt{S}^n, \pi\big) \leq \frac{\pi \sqrt{2} (d+1)}{(d 2^{3/2})^{1/d}} \cdot \left(\left(1 - e^{\varphi(1) - \varphi(0)}\right)\cdot \frac{d}{d+1}\right)^n
		W_{\mathrm{d}}\big(\mu , \pi), \qquad \forall n \in \mathbb{N}, \mu \in \mathcal{M}_{\mathbb{S}^d}.
	\]
\end{example}

\appendix
\section*{Appendix}

The appendix collects auxiliary technical statements.

\section{Auxiliary lemma for Section \ref{S: Gibbs sampler}}

\begin{lemma}\label{L: Tensor product Wasserstein distance}
	Let $(\mathsf{A}, \mathcal{B}(\mathsf{A}))$, $(\mathsf{B}, \mathcal{B}(\mathsf{B}))$ and $(\mathsf{C}, \mathcal{B}(\mathsf{C}))$ be three Polish spaces equipped with lower semi-continuous metrics $\mathrm{d}_{\mathsf{A}}$, $\mathrm{d}_{\mathsf{B}}$ and $\mathrm{d}_{\mathsf{C}}$, respectively.
	Fix $\mathfrak{p} \in [1, \infty]$, and denote by $\mathrm{d}_{\mathsf{B} \times \mathsf{C}}$ the $\mathfrak{p}$-product metric of $\mathrm{d}_{\mathsf{B}}$ and $\mathrm{d}_{\mathsf{C}}$ on $\mathsf{B} \times\mathsf{C}$.
	If $S: \mathsf{A} \times \mathcal{B}(\mathsf{B}) \to [0,1]$ and $T:\mathsf{B} \times \mathcal{B}(\mathsf{C})\to[0,1]$ are Markov kernels
	such that there exist $q \in [1, \infty)$ and $K_1, K_2 \in [0, \infty)$ with
	\begin{equation}\label{Eq: tensor product Wasserstein distance condition}
		W_{\mathrm{d}_{\mathsf{B}}}^{(q)}\big(S(x_1, \cdot), S(x_2, \cdot)\big)\leq K_1 \cdot \mathrm{d}_{\mathsf{A}}(x_1, x_2), \quad \text{and} \quad
		W_{\mathrm{d}_{\mathsf{C}}}^{(q)}\big(T(y_1, \cdot), T(y_2, \cdot)\big)\leq K_2 \cdot \mathrm{d}_{\mathsf{B}}(y_1, y_2)
	\end{equation}
	for all $x_1,x_2 \in \mathsf{A}$, $y_1,y_2 \in \mathsf{B}$,
	then we have
	\[
		W_{\mathrm{d}_{\mathsf{B} \times \mathsf{C}}}^{(q)}\big((S \otimes T)(x_1, \cdot), (S\otimes T)(x_2, \cdot)\big)\leq
		K_1 (K_2 + 1)\cdot \mathrm{d}_{\mathsf{A}}(x_1,x_2), \qquad \forall x_1,x_2 \in \mathsf{A}.
	\]
\end{lemma}
\begin{proof}
	Fix $x_1,x_2 \in \mathsf{A}$. Let $\eta \in \mathcal{C}(S(x_1, \cdot), S(x_2, \cdot))$ be on optimal coupling with respect to $W_{\mathrm{d}_{\mathsf{B}}}^{(q)}$.
	Moreover by \cite[Theorem 1.1]{zhang2000existence}, let $\Xi: \mathsf{B}^2 \times \mathcal{B}(\mathsf{C}^2) \to [0,1]$
	be a Markov kernel such that for all $y_1,y_2 \in \mathsf{B}$ the probability measure $\Xi((y_1,y_2), \cdot)$ is an optimal coupling of $T(y_1, \cdot)$ and $T(y_2, \cdot)$ with respect to $W_{\mathrm{d}_{\mathsf{C}}}^{(q)}$.
	Set
	\[
	\xi(\mathsf{D}) := \int_{\mathsf{B}^2} \int_{\mathsf{C}^2} \mathbbm{1}_{\mathsf{D}}\big((y_1,z_1,y_2,z_2)\big)\ \Xi\big((y_1,y_2), \d (z_1, z_2)\big)\, \eta\big(\d(y_1,y_2)\big), \qquad \mathsf{D} \in \mathcal{B}\big((\mathsf{B}\times\mathsf{C})^2\big).
	\]
	We have $\xi \in \mathcal{C}((S \otimes T)(x_1, \cdot),(S\otimes T)(x_2, \cdot))$.
	Hence using Minkowski inequality and the optimality of $\eta$ and $\Xi((y_1,y_2), \cdot)$
	\begin{align*}
		&W_{\mathrm{d}_{\mathsf{B} \times \mathsf{C}}}^{(q)}\big((S \otimes T)(x_1, \cdot), (S\otimes T)(x_2, \cdot)\big)\\
		&\qquad \overset{\eqref{Eq: product metric ineqaulity}}{\leq} \left(\int_{(\mathsf{B}\times\mathsf{C})^2} \left(\mathrm{d}_{\mathsf{B}}(y_1, y_2) + \mathrm{d}_{\mathsf{C}}(z_1, z_2) \right)^{q}\
		\xi\big(\d(y_1,z_1,y_2,z_2)\big)\right)^{1/q}\\
		& \qquad \leq  
			\left(\int_{\mathsf{B}^2} \mathrm{d}_{\mathsf{B}}(y_1, y_2)^{q}\
			\ \eta\big(\d(y_1,y_2)\big)\right)^{1/q}
			+  \left(\int_{\mathsf{B}^2} W_{\mathrm{d}_{\mathsf{C}}}^{(q)}\big(T(y_1, \cdot), T(y_2, \cdot)\big)^{q}\ \eta\big(\d(y_1,y_2)\big)\right)^{1/q}\\
		& \qquad \overset{\eqref{Eq: tensor product Wasserstein distance condition}}{\leq} W_{\mathrm{d}_{\mathsf{B}}}^{(q)} \big(S(x_1, \cdot), S(x_2, \cdot)\big)
			+  K_2\cdot\left(\int_{\mathsf{B}^2}\mathrm{d}_{\mathsf{B}}(y_1,y_2)^{q}\ \eta\big(\d(y_1,y_2)\big)\right)^{1/q}\\
		&\qquad= (K_2 + 1) W_{\mathrm{d}_{\mathsf{B}}}^{(q)} \big(S(x_1, \cdot), S(x_2, \cdot)\big)
		\overset{\eqref{Eq: tensor product Wasserstein distance condition}}{\leq} K_1 (K_2 + 1) \mathrm{d}_{\mathsf{A}}(x_1,x_2).
	\end{align*}
\end{proof}

\section{Auxiliary statements for Section \ref{S: Slice sampling}}

\begin{remark}\label{R: How to apply Rueschendorf}
	In this remark we discuss how to apply \cite[Theorem 2.a)]{rueschendorf1985Wasserstein} to obtain \eqref{Eq: Slice sampler Wasserstein ergodicity} and \eqref{Eq: explicit expression 1-dim Wasserstein distance}.
	For simplicity, we use $g$ and $m$ to denote both the functions $g_{\alpha, \kappa}$ and $g$, and $m_{\alpha, \kappa}$ and $m$ from Proposition \ref{P: Wasserstein ergodicity of slice sampling} and Lemma \ref{L: explict formula for m_alpha-Wasserstein distance}, respectively.

	Let $\eta, \xi \in \mathcal{M}_{\mathbb{R}}$ be concentrated on $p(\mathsf{X})$.
	Our aim is to establish the identity
	\begin{equation}\label{Eq: How to apply Rueschendorf}
		W_{m}\big(\eta, \xi\big) = W_{|\cdot|} \Big((g\circ \ell)(\eta), (g\circ \ell)(\xi)\Big).
	\end{equation}
	Observe that a direct application of \cite[Theorem 2.a)]{rueschendorf1985Wasserstein} to obtain \eqref{Eq: How to apply Rueschendorf} would require $g\circ \ell$ to have domain $\mathbb{R}$.
	Instead, we artificially extend $g \circ \ell$ to the function
	$h: \mathbb{R} \to \mathbb{R}$ with $ h= (g \circ \ell)$ on $p(\mathsf{X})$ and $h \equiv 1$ on $\mathbb{R}\setminus p(\mathsf{X})$,
	and define $\tilde{m}(x,y):= |h(x) - h(y)|, x,y \in \mathbb{R}$.
	Because any coupling of $\eta$ and $\xi$ is concentrated on $p(\mathsf{X})^2$,
	we have $W_{m}(\eta, \xi) = W_{\tilde{m}}(\eta, \xi)$.
	Moreover, 
	\[
		h^{-1}\big((- \infty,x]\big)
		= \big\{s \in p(\mathsf{X}) \mid  (g \circ \ell)(s) \leq x \big\} \cup \{s \in \mathbb{R} \setminus p(\mathsf{X}) \mid 1 \leq x \},
		\qquad x \in \mathbb{R}.
	\]
	Observe that $\mathbb{R}\setminus p(\mathsf{X})$ is a null set for $\eta$,
	such that the cumulative distribution function of $(g\circ \ell)(\eta)$ and $h(\eta)$ coincide.
	The same holds for $\xi$.
	Combining \cite[Theorem 2.a)]{rueschendorf1985Wasserstein} for $h$ and $\tilde{m}$ with these identities yields \eqref{Eq: How to apply Rueschendorf}.
\end{remark}

As before, denote by $\mathbb{S}^{d}$ the $d$-dimensional Euclidean unit sphere, $d$-dimensional hyperbolic space as $\mathcal{H}_d$,
and by $\omega_{d}$ the volume of $\mathbb{S}^d$.
Moreover, we use $B_r(x)$ to denote the open metric ball with radius $r \in[0, \infty)$ and centre point $x$ in the respective metric space. 
.

\begin{lemma}\label{L: General level set computations}
	Let $\mathsf{X} \in \{\mathbb{R}^d, \mathbb{S}^{d}, \mathcal{H}_d\}$ for some $d \in \mathbb{N}$,
	equipped with the metric $\mathrm{d}$ induced by the Riemannian metric.
	Let $o \in \mathsf{X}$,
	and denote by $\nu$ the Riemannian measure on $(\mathsf{X},\mathcal{B}(\mathsf{X}))$.
	Moreover set
	\[
	c := \begin{cases}
		\infty, & \mathsf{X} \in \{\mathbb{R}^d, \mathcal{H}_d\},\\
		\pi, & \mathsf{X} = \mathbb{S}^{d}.
	\end{cases}	
	\qquad \text{and} \qquad
	s_{\mathsf{X}}(u) := \begin{dcases}
		u^{d-1}, & \mathsf{X} = \mathbb{R}^d,\\
		\sinh(u)^{d-1}, & \mathsf{X} = \mathcal{H}_d,\\
		\sin(u)^{d-1}, & \mathsf{X} = \mathbb{S}^{d},
	\end{dcases}
	\]
	and $\Psi(x) := \int_{(0,x)} s_{\mathsf{X}}(u)  \ \lambda_1(\d u)$ for $x \in [0,c)$.
	Then $k(r):=\nu(B_{r}(o)) = \Psi(r)\omega_{d-1}$ for all $r \in [0,c)$, and for all $r_1, r_2 \in [0,c)$ holds
	\[
	W_{\mathrm{d}}\left(\frac{1}{k(r_1)}\nu\vert_{B_{r_1}(o)}, \frac{1}{k(r_2)}\nu\vert_{B_{r_2}(o)}\right)
	\leq \left| \frac{\int_{(0,r_1)} u s_{\mathsf{X}}(u)\ \lambda_1(\d u)}{\Psi(r_1)}
	- \frac{\int_{(0,r_2)} u s_{\mathsf{X}}(u)\ \lambda_1(\d u)}{\Psi(r_2)} \right|.
	\]
\end{lemma}
\begin{proof}
	Denote by $\sigma$ the volume measure on $\mathbb{S}^{d-1}$.
	Using e.g.\ \cite[Corollary 10.17]{lee2018introduction}, we have
	\begin{align*}
		k(r):=\nu(B_{r}(o))
		= \int_{\mathbb{S}^{d-1}} \int_{[0,r)}  s_{\mathsf{X}}(u) \ \lambda_1(\d u)\, \sigma(\d \theta) 
		= \Psi(r) \omega_{d-1}, \qquad r \in[0, c).
	\end{align*}
	
	Denote by $\mathrm{Exp}_o: \mathbb{R}^d \to \mathsf{X}$ the exponential function at $o$ 
	(where we choose an isometry between the tangent space $T_o\mathsf{X}$ to $\mathsf{X}$ at $o$ and $\mathbb{R}^d$ to identify $T_o\mathsf{X}$ with $\mathbb{R}^d$).
	For $r_1,r_2 \in[0, c)$ with $r_1 \geq r_2$ define the probability measure
	\[
	\eta(\mathsf{C}) := \frac{1}{\omega_{d-1}}\int_{\mathbb{S}^{d-1}} \int_{(0,1)}\mathbbm{1}_{\mathsf{C}}\Big(\mathrm{Exp}_o\big(\Psi^{-1}(\Psi(r_1)s)\theta\big),
	\mathrm{Exp}_o\big(\Psi^{-1}(\Psi(r_2)s)\theta\big)\Big)\ \lambda_1(\d s) \, \sigma(\d \theta)
	\]
	for $\mathsf{C} \in \mathcal{B}(\mathsf{X}^2)$.
	Then $\eta \in \mathcal{C}(1/k(r_1)\nu\vert_{B_{r_1}(o)}, 1/k(r_2)\nu\vert_{B_{r_2}(o)})$,
	because using the substitution $y = \Psi^{-1}(\Psi(r_1)s)$ together with the inverse function theorem and \cite[Corollary 10.17]{lee2018introduction}, we have for $\mathsf{A} \in \mathcal{B}(\mathsf{C})$
	\begin{align*}
		\eta(\mathsf{A} \times \mathsf{X})
		&= \frac{1}{\omega_{d-1}}\int_{\mathbb{S}^{d-1}} \int_{(0,1)}\mathbbm{1}_{\mathsf{A}}\Big(\mathrm{Exp}_o\big(\Psi^{-1}(\Psi(r_1)s)\theta\big)\Big)
		\ \lambda_1(\d s) \, \sigma(\d \theta)\\
		&= \frac{1}{k(r_1)}\int_{\mathbb{S}^{d-1}} \int_{(0,r_1)}\mathbbm{1}_{\mathsf{A}}\Big(\mathrm{Exp}_o\big(y\theta\big)\Big) s_{\mathsf{X}}(y)
		\ \lambda_1(\d y) \, \sigma(\d \theta)
		= \frac{1}{k(r_1)} \nu\big(\mathsf{A} \cap B_{r_1}(o)\big),
	\end{align*}
	and a similar argument also works for the second marginal of $\eta$.
	Observe that for all $\theta \in \mathbb{S}^{d-1}$ the  geodesic $[0, c) \to \mathsf{X}, s \mapsto \mathrm{Exp}_o(s)$ is minimizing.
	Therefore 
	\begin{align*}
		&W_{\mathrm{d}}\left(\frac{1}{k(r_1)}\nu\vert_{B_{r_1}(o)}, \frac{1}{k(r_2)}\nu\vert_{B_{r_2}(o)}\right)
		\leq  \int_{(0,1)}\big|\Psi^{-1}\big(\Psi(r_1)s\big) - \Psi^{-1}\big(\Psi(r_2)s\big)\big|\ \lambda_1(\d s) \\
		&\qquad= \frac{1}{\Psi(r_1)} \int_{(0,\Psi(r_1))}\Psi^{-1}(s) \ \lambda_1(\d s)
		- \frac{1}{\Psi(r_2)} \int_{(0,\Psi(r_2))}\Psi^{-1}(s) \ \lambda_1(\d s)\\
		&\qquad=  \frac{1}{\Psi(r_1)} \int_{(0,r_1)} y  s_{\mathsf{X}}(y) \ \lambda_1(\d y)
		- \frac{1}{\Psi(r_2)} \ \int_{(0,r_2)} y  s_{\mathsf{X}}(y) \ \lambda_1(\d y),
	\end{align*}
	where for the last equality we used the substitution $y = \Psi^{-1}(s)$ together with the inverse function theorem.
\end{proof}

\begin{lemma}\label{L: concavity considerations}
	Let $h: \mathsf{J} \to \mathsf{I}$ be a strictly monotone and $\phi: \mathsf{I} \to \mathbb{R}$ a monotone function on intervals $\mathsf{J}\subseteq[0, \infty)$ and $\mathsf{I}\subseteq \mathbb{R}$.
	If $\phi$ is concave, and either $h$ is increasing and $\phi$ is decreasing, or $h$ is decreasing and $\phi$ is increasing, then
	\[
		\left|\phi\big(h(st)\big) - \phi\big(h(sr)\big)\right| \leq \frac{h(st)- h(sr)}{h(t) - h(r)} \cdot \left|\phi\big(h(t)\big) - \phi\big(h(r)\big)\right|
	\]
	for all $s \in (0,1)$, $r,t \in \mathsf{J}$ with $r \neq t$ and  $sr, st \in \mathsf{J}$.
\end{lemma}
\begin{proof}
	Let $s \in (0,1)$ and $r,t \in \mathsf{J}$ with $t < r$ and $sr,st \in \mathsf{J}$.
	
	\textit{1. case: $h$ increasing, $\phi$ decreasing.}
	Since $\phi$ is concave and $h(sr) \leq h(r)$, $h(st) \leq h(t)$, we have
	\[
		\frac{\phi\big(h(st)\big)- \phi\big(h(sr)\big)}{h(st)- h(sr)} \geq \frac{\phi\big(h(t)\big)- \phi\big(h(r)\big)}{h(t)- h(r)}.
	\]
	The statement follows, because $h(st) -h(sr) < 0$ and $\phi \circ h$ is decreasing.
	
	\textit{2. case: $h$ decreasing, $\phi$ increasing.}
	Using again that $\phi$ is concave and this time $h(r) \leq h(sr)$, $h(t) \leq h(st)$, we get
	\[
		\frac{\phi\big(h(st)\big)- \phi\big(h(sr)\big)}{h(st)- h(sr)} \leq \frac{\phi\big(h(t)\big)- \phi\big(h(r)\big)}{h(t)- h(r)}.
	\]
	The statement follows, because $h(st) -h(sr) > 0$ and $\phi \circ h$ is decreasing.
\end{proof}

\begin{lemma}\label{L: bound for balls in hyperbolic space}
	Let $d \geq 2$. Setting $\Psi: [0, \infty)\to [0, \infty), x \mapsto \int_{0}^{x} \sinh^{d-1}(u)\ \d u$, we have
	\[
		\left|\frac{\int_{0}^{x} u \sinh^{d-1}(u)\ \d u}{\Psi(x)} - \frac{\int_{0}^{y} u \sinh^{d-1}(u)\ \d u}{\Psi(y)}\right|
		\leq \left|\ln\big(\Psi(x)\big) - \ln\big(\Psi(y)\big) \right|, \qquad \forall x,y \in [0, \infty).
	\]
\end{lemma}
\begin{proof}
	It suffices to show that the function
	\[
		h: [0, \infty) \to \mathbb{R}, \qquad x \mapsto\frac{\int_{0}^{x} u \sinh^{d-1}(u)\ \d u}{\Psi(x)}  - \ln\big(\Psi(x)\big)
	\]
	is decreasing.
	To this end, consider the function $g: [0, \infty) \to \mathbb{R}, x \mapsto (x-1)\Psi(x) - \int_{0}^{x} u \sinh^{d-1}(u)\ \d u$.
	It has derivatives $g'(x) = \Psi(x) -  \sinh^{d-1}(x)$ and 
	\begin{align*}
		g''(x) &= \sinh^{d-2}(x) \big(\sinh(x) - (d-1) \cosh(x)\big) \leq \sinh^{d-2}(x) \big(\sinh(x)-\cosh(x)\big) \\
			&= -\sinh^{d-2}(x)e^{-x} \leq 0, \qquad\qquad  \forall x \in [0, \infty).
	\end{align*}
	Therefore $g'$ is decreasing, and $g'(0) = 0$ implies $g' \leq 0$.
	Hence also $g$ is decreasing. Together with $g(0) = 0$, we get $g \leq 0$.
	Since the derivative of $h$ is given by $h'(x) = \sin^{d-1}(x)g(x)/\Psi(x)^2 $, $x \in [0, \infty)$, this implies $h' \leq 0$.
\end{proof}

\begin{lemma}\label{L: bounds for balls on the sphere}
	Let $d \geq 2$, and set $\Psi: [0, \pi]\to [0, \infty), x \mapsto \int_{0}^{x} \sin^{d-1}(u)\ \d u$.
	Then
	\[
		\left|\frac{\int_0^{x} u \sin^{d-1}(u)\ \d u}{\Psi(x)} - \frac{\int_0^{y} u \sin^{d-1}(u)\ \d u}{\Psi(y)}\right|
		\leq \frac{d\pi\sqrt{2}}{2^{3/(2d)}} \cdot \left|\Psi(x)^{1/d} - \Psi(y)^{1/d}\right|, \qquad \forall x,y \in [0,\pi].
	\]
	Moreover, $\phi:[0, \pi] \to [0, \infty), x \mapsto \Psi(x)^{1/d}$ is $d^{-1/d}$-Lipschitz continuous and concave.
\end{lemma}
\begin{proof}
	For the first claim, it suffices to show that the function 
	\[
		h: \left[0, \sqrt[d]{\Psi(\pi)}\right] \to [0, \infty), \qquad x \mapsto \frac{\int_0^{\Psi^{-1}(x^d)} u \sin^{d-1}(u)\ \d u}{x^d}
	\]
	is $d\pi/\Psi(\pi)^{1/d}$-Lipschitz continuous,
	since
	\[
		\Psi(\pi) = \int_{0}^{\pi/2} \sin^{d-1}(x)\ \d x + \int_{0}^{\pi/2} \cos^{d-1}(x)\ \d x \geq 2 \cdot\frac{(d-2)!!}{(d-1)!!}
		\geq 2^{(3-d)/2},
	\]
	see \cite{meserve1948double}, where $n!!$ for $n \in \mathbb{N}$ denotes the double factorial.
	We have $h'(x) = d\Psi^{-1}(x^d)/x - d h(x)/x$. 
	In particular $h' \geq 0$  and
	\begin{align*}
		\max_{x \in \left [0 , \sqrt[d]{\Psi(\pi)}\right ]} h'(x) 
		\leq \max_{x \in \left [0, \sqrt[d]{\Psi(\pi)}\right ]} d\cdot\frac{\Psi^{-1}(x^d)}{x}
		= \max_{z \in [0,\pi]} d\cdot \left(\frac{z^d}{\Psi(z)}\right)^{1/d}
		\leq \frac{d\pi}{\Psi(\pi)^{1/d}},
	\end{align*}
	since the function $q:[0,\pi] \to [0, \infty), z \mapsto z^d/\Psi(z)$ is increasing:
	Note that $\tilde{q}: [0, \pi] \to [0, \infty), z \mapsto d \Psi(z) - z\sin^{d-1}(z)$ satisfies $\tilde{q}\geq 0$,
	because $\tilde{q}'(z) = (d-1)\sin^{d-2}(z)(\sin(z) - z \cos(z)) \geq 0$ and $\tilde{q}(0) = 0$.
	Therefore $q'(z) =  \tilde{q}(z)  z^{d-1}/\Psi(z)^2\geq 0$ for all $z \in [0, \pi]$.
	
	For the second claim, note that we have, 
	\[
		\phi''(x) = \frac{(d-1) \sin^{d-2}(x)}{d\,\Psi(x)^{2-1/d}} \cdot \left (\cos(x) \Psi(x) - \frac{1}{d}\sin^d(x)\right ) \leq 0, \qquad x \in [0,\pi],
	\]
	since $\tilde{f}: [0,\pi], x \mapsto \cos(x) \Psi(x) - \frac{1}{d}\sin^d(x)$ satisfies $\tilde{f}'(x) = -\sin(x) \Psi(x) \leq 0$ for $x \in [0,\pi]$ and $\tilde{f}(0) =0$.
	This implies that $\phi$ is concave.
	Moreover, using l'Hosptial
	\begin{align*}
		|\phi(x)'| &= \frac{1}{d} \cdot \frac{\sin^{d-1}(x)}{\Psi(x)^{(d-1)/d}}
		 \leq \frac{1}{d} \lim_{x \to 0} \frac{\sin^{d-1}(x)}{\Psi(x)^{(d-1)/d}}
		  =\frac{1}{d} \left(\lim_{x \to 0} \frac{\sin^d(x)}{\Psi(x)}\right)^{(d-1)/d}\\
		  &=\frac{1}{d} \left(\lim_{x \to 0} \frac{d\sin^{d-1}(x)\cos(x)}{\sin^{d-1}(x)}\right)^{(d-1)/d}
		  = d^{-1/d}.
	\end{align*}
\end{proof}

\section*{Acknowledgements}
The author would like to thank Daniel Rudolf for valuable feedback.
Moreover, the author gratefully acknowledges support from the Deutsche Forschungsgemeinschaft (DFG, German Research Foundation) – project number 522337282.

 \bibliography{sources.bib}  
\end{document}